\documentclass{article}
\usepackage[left=1in,right=1in,top=1in,bottom=1in]{geometry}

\usepackage{natbib}
 \bibpunct[, ]{(}{)}{,}{a}{}{,}%

\usepackage[utf8]{inputenc}

\usepackage{algorithm}
\usepackage{algorithmicx}
\usepackage{algpseudocode}
\usepackage{caption}
\usepackage{graphicx}
\usepackage{subcaption}
\usepackage{enumitem}

\usepackage{amsmath}
\usepackage{amsthm}
\usepackage{thmtools}
\usepackage{amssymb}
\usepackage{mathtools}
\usepackage{xcolor}
\usepackage{todonotes}
\usepackage{xcolor}
\usepackage{booktabs}
\usepackage{siunitx}

\usepackage{diagbox}
\usepackage[toc,page]{appendix}

\usepackage{tikz} 
\usetikzlibrary{fit}
\usepackage[colorlinks=true]{hyperref}
\usepackage[nameinlink]{cleveref}

\newtheorem{theorem}{Theorem}
\newtheorem{remark}[theorem]{Remark}
\newtheorem{lemma}[theorem]{Lemma}
\newtheorem{proposition}[theorem]{Proposition}
\newtheorem{assumption}[theorem]{Assumption}
\newtheorem{definition}[theorem]{Definition}

\crefname{algorithm}{Algorithm}{Algorithms}
\crefname{theorem}{Theorem}{Theorems}
\crefname{lemma}{Lemma}{Lemmas}
\crefname{assumption}{Assumption}{Assumptions}
\crefname{proposition}{Proposition}{Propositions}
\crefname{figure}{Figure}{Figures}
\Crefname{figure}{Figure}{Figures}
\crefname{table}{Table}{Tables}
\Crefname{table}{Table}{Tables}

\newcommand*\dd{\mathop{}\!\mathrm{d}}
\DeclareMathOperator*{\TV}{TV}
\DeclareMathOperator*{\TR}{TR}
\DeclareMathOperator*{\LR}{LR}
\DeclareMathOperator*{\conv}{conv}

\newcommand{\N}{\mathbb{N}}
\newcommand{\Z}{\mathbb{Z}}
\newcommand{\R}{\mathbb{R}}

\newcommand{\tripjref}[1]{\hyperref[{eq:tr-ip-j}]{(TR-IP({#1}))}}

\usepackage{crossreftools}

\makeatletter
\newcommand{\optionaldesc}[2]{%
  \phantomsection
  #1\protected@edef\@currentlabel{#1}\label{#2}%
}
\makeatother

\DeclareMathOperator*{\argmin}{arg\,min}
\DeclareMathOperator*{\argmax}{arg\,max}

\title{Efficient Solution of Discrete Subproblems
Arising in Integer Optimal Control with Total Variation Regularization}
\author{Marvin Severitt
\thanks{TU Dortmund University
        (\texttt{marvin.severitt@tu-dortmund.de})}
        \and
Paul Manns
\thanks{TU Dortmund University
        (\texttt{paul.manns@tu-dortmund.de})}}

\begin{document}
\maketitle
\begin{abstract}
We consider a class of integer linear programs (IPs) that arise
as discretizations of trust-region 
subproblems of a trust-region algorithm for the solution of control problems, where the control input is an
integer-valued function on a one-dimensional domain
and is regularized with a total variation
term in the objective, which may be interpreted as a
penalization of switching costs between different control modes.

We prove that solving an instance of the considered problem class is 
equivalent to solving a resource constrained shortest path problem
(RCSPP) on a layered directed acyclic graph. This structural
finding yields an algorithmic solution approach based on topological 
sorting and corresponding run time complexities that are quadratic in 
the number of discretization intervals of the underlying control problem,
the main quantifier for the size of a problem instance.
We also consider the solution of the RCSPP with an $A^*$ algorithm.
Specifically, the analysis of a Lagrangian relaxation yields a
consistent heuristic function for the $A^*$ algorithm and a
preprocessing procedure, which can be employed to accelerate the
$A^*$ algorithm for the RCSPP without losing optimality of the
computed solution.

We generate IP instances by executing the trust-region algorithm on several 
integer optimal control problems. The numerical results show that the
accelerated $A^*$ algorithm and topological sorting outperform a general
purpose IP solver significantly. Moreover, the accelerated $A^*$ algorithm
is able to outperform topological sorting for larger problem instances.
\end{abstract}

\section{Introduction}

Integer optimal control problems (IOCPs)---optimization problems
constrained by (partial) differential equations with
integer-valued control input functions---are versatile
modeling tools with diverse applications from topology 
optimization, see e.g.\
\cite{svanberg2007sequential,haslinger2015topology,liang2019topology,leyffer2021convergence}, over 
energy management of buildings, see e.g.\
\cite{zavala2010proactive},
to network transportation problems, e.g.\ traffic flow as considered in
\cite{goettlich2014optimization,goettlich2017partial} or gas flow
as considered in
\cite{pfetsch2015validation,hante2017challenges,habeck2019global}.

A computationally efficient algorithmic solution approach that provides optimal
approximation properties for many 
IOCPs under appropriate assumptions
on the underlying differential equations is the
so-called combinatorial integral approximation, see e.g.\
\cite{sager2011combinatorial,sager2012integer,hante2013relaxation,jung2015lagrangian,manns2020multidimensional,kirches2021compactness}.

The optimality principle underlying the combinatorial integral approximation
comes at the cost of high-frequency switching of the resulting control input
function between different integers, see Figure 5 in 
\cite{manns2020multidimensional}
or Figures 3 and 4 \cite{kirches2021compactness}, which impairs the implementability
of the resulting controls. This behavior cannot be avoided
if an optimal control function for the continuous relaxation is not
already integer-valued.
Recent research has produced results to alleviate this problem by promoting
integer-valued control functions in the continuous relaxation, see
\cite{manns2021relaxed}, and minimizing switching costs
while maintaining
approximation guarantees in the combinatorial integral approximation,
see \cite{bestehorn2019switching,bestehorn2020mixed,bestehorn2020matching,bestehorn2021switching,sager2021mixed}.

However, even combining both approaches cannot avoid this problem in many situations
as is pointed out and visualized in sections 4 and 5 in \cite{manns2021relaxed},
which motivates to seek alternative approaches to mitigate high-frequency
oscillations in IOCPs. 
\cite{leyffer2021sequential} propose a novel trust-region algorithm, which---after 
discretization of the underlying differential equation---produces
a sequence of integer linear programs (IPs) as trust-region subproblems
for control functions that are defined on one-dimensional domains.
Switching costs are modeled with a penalty in the objective and
not approximated but modeled exactly with linear inequalities in
each subproblem, thereby yielding a structure-preserving algorithm. The
subproblems are solved with a general purpose IP solver in 
\cite{leyffer2021sequential}, which yields run times that are several orders of magnitude
higher than those of the combinatorial integral approximation, specifically
the approach proposed in \cite{bestehorn2020mixed},
where switching costs in the form of penalty terms are considered
in the approximation problems. This shortcoming is addressed in this work.

\paragraph{Contribution}
Minimizing switching costs in the combinatorial integral approximation with
penalty terms are treated with a shortest path approach in 
\cite{bestehorn2020mixed}. We transfer these ideas to the trust-region 
subproblems that arise in algorithm proposed in \cite{leyffer2021sequential}
and prove equivalence to resource constrained shortest path problems (RCSPPs)
on layered directed acyclic graphs (LDAGs).

We provide run time estimates of a topological sorting approach and analyze
an $A^*$ algorithm for the resulting RCSPPs, which yields that both
approaches are  pseudo-polynomial solution algorithms and the problem is 
fixed-parameter tractable. In particular, we show that the subproblems are 
generally NP-hard but the NP-hardness stems from the set of integers that 
constitute the admissible values for the control function of the underlying
IOCP. This set is constant over all subproblems of one run of the trust-region
algorithm and generally small, usually containing only between 2 and 10
values depending on the application.

We analyze a Lagrangian relaxation of the IP formulation that yields upper and
lower bounds on subpaths in the RCSPP formulation and provides a consistent and
monotone heuristic, which yields that an $A^*$ algorithm can be accelerated 
without losing optimality of the computed path. We also analyze a dominance 
principle to further reduce the search space. We evaluate the approach by 
executing the trust-region algorithm on instances of two classes of IOCPs.
We solve the generated subproblems with different algorithmic approaches, 
specifically topological sorting, the general purpose IP solver SCIP, see 
\cite{GamrathEtal2020ZR}, and the $A^*$ algorithm with the aforementioned
accelerations. The general purpose IP solver is several orders of 
magnitude slower than the other two approaches. The $A^*$ algorithm is
able to outperform topological sorting for larger sizes of the
subproblems, thereby suggesting to choose the subproblem solver
depending on the subproblem size.

\paragraph{Structure of the Remainder}
We introduce some notation that we use in the remainder of the
manuscript below. Then we introduce the superordinate IOCPs and the
trust-region algorithm in \S\ref{sec:trust-region alg}.
We continue by presenting and analyzing the IPs that arise as discretized 
trust-region subproblems and which are the main object of our investigation
in \S\ref{sec:discretized_subproblem}, where we also show their
equivalence to RCSPPs on LDAGs. We show how the $A^*$ algorithm
may be accelerated using information obtained from the Lagrangian
Relaxation in \S\ref{sec:lagrangian}. The set of our computational
experiments is described \S\ref{sec:computational_experiments}. The
results are presented in \S\ref{sec:numresults}.
We draw a conclusion in \S\ref{sec:conclusion}.

\paragraph{Notation}
For a natural number $n \in \N$ we define the notation
$[n] \coloneqq \{1,\ldots,n\}$. For a scalar $a \in \R$ and a
set $B \subset \R$ we abbreviate $B - a \coloneqq \{b - a\,|\, b \in B\}$.

\section{IOCPs and Trust-region Algorithm}
\label{sec:trust-region alg}
We briefly state the trust-region algorithm and the class of IOCPs such that the analyzed problems
arise as trust-region subproblems after discretization. The general IOCP reads
\begin{gather}\label{eq:p}
\begin{aligned}
\min_{x \in L^2(0,T)} & F(x) + \alpha \TV(x) \eqqcolon J(x) \\
\text{~~s.t.~~} &
x(t) \in \Xi \text{ for almost all (a.a.) } t \in (0,T),
\end{aligned}
\tag{IOCP}
\end{gather}
where $T > 0$, $L^1(0,T)$ and $L^2(0,T)$ denote the spaces of integrable and square-integrable
functions on $(0,T)$, $F : L^2(0,T) \to \R$ denotes the main part of the objective that abstracts
from the underlying differential equation, $x$ is the optimized $\Xi$-valued control function,
and $\Xi \subset \Z$ is a finite set of integers.
The term $\alpha \TV(x)$ models the switching costs of the function $x$, where $\alpha > 0$ is a positive
scalar and $\TV(x) \in [0,\infty]$ is the total variation of $x$, which amounts to the sum of the jump heights of
$x$ because $x$ takes only finitely many values, see \cite{leyffer2021sequential}.

The trust-region subproblem that is proposed in \cite{leyffer2021sequential} reads
\begin{gather}\label{eq:tr}
\begin{aligned}
\min_{d \in L^1(0,T)} & (\nabla F(x), d)_{L^2(0,T)} + \alpha \TV(x + d) - \alpha\TV(x)
\eqqcolon \ell(x, d) \\
\text{~~s.t.~~}
& x(t) + d(t) \in \Xi \text{ for a.a.\ } t \in (0,T),\\
& \|d\|_{L^1(0,T)} \le \Delta
\end{aligned}
\tag{TR}
\end{gather}
for a trust-region radius $\Delta > 0$. We denote the instance of
\eqref{eq:tr} for a given control function
$x$ and a given trust-region radius $\Delta > 0$ by $\TR(x,\Delta)$.

The trust-region algorithm is given in \cref{alg:slip} and works as follows. In
every outer iteration (indexed by $n$) the trust-region radius is
reset to the input $\Delta^0 > 0$ and
an inner loop (iteration index $k$) is triggered. In the inner iteration the
trust-region subproblem is solved for the trust-region radius. Then the algorithm terminates
if the predicted reduction (objective of the trust-region subproblem) is zero for a positive 
trust-region radius, which implies a necessary optimality condition under appropriate assumptions
on the function $F$ according to \cite{leyffer2021sequential}.
If the computed iterate can be accepted---that is, if the actual reduction is at least a fraction
of the linearly predicted reduction---the inner loop terminates and the
current iterate is updated. If the step cannot be accepted (is rejected),
then the trust-region is halved and the next inner iteration of the inner loop is triggered.
\begin{algorithm}[H]
\caption{Sketch of the trust-region algorithm from \cite{leyffer2021sequential}}\label{alg:slip}
\textbf{Input: } $x^0$ feasible for \eqref{eq:p}, $\Delta^0 > 0$, $\rho \in (0,1)$
\begin{algorithmic}[1]
\For{$n = 1,\ldots$}
\State $k \gets 0$, $\Delta^{n,0} \gets \Delta^0$
\Repeat
\State $d^{n,k} \gets$ minimizer of $\TR(x^{n-1},\Delta^{n,k})$
\Comment{Compute step.}
\If{$\ell(x^{n-1}, d) = 0$}
\Comment{The predicted reduction is zero.}
\State Terminate with solution $x^{n-1}$.
\ElsIf{$J(x^{n - 1}) - J(x^{n-1} + d^{n,k}) < \rho \ell(x^{n-1},d^{n,k})$}
\Comment{Reject step.}
\State $\Delta^{n,k+1} \gets \Delta^{n,k} / 2$, $k \gets k + 1$
\Else
\Comment{Accept step.}
\State $x^n \gets x^{n-1} + d^{n,k}$, $k \gets k + 1$
\EndIf
\Until{$J(x^{n - 1}) - J(x^{n-1} + d^{n,k}) \ge \sigma \ell(x^{n-1},d^{n,k})$}
\EndFor
\end{algorithmic}
\end{algorithm}

\section{The Discretized Trust-region Subproblem}\label{sec:discretized_subproblem}
We provide and explain the IP formulation of the considered problem class in \S\ref{sec:ip_formulation} for which we show the NP-hardness in \S\ref{sec:np-hardness}. Then we construct an LDAG and prove that a shortest path search 
on it is equivalent to solving the IP formulation in \S\ref{sec:spp_reformulation}.
In \S\ref{sec:rcspp_formulation} we factorize the graph by means of an equivalence
relation and obtain the aforementioned equivalence to an RCSPP
on the resulting reduced graph (quotient graph).
In \S\ref{sec:solution algo} we present two algorithms to solve the SPP on the LDAG, which are accelerated
in \S\ref{sec:lagrangian} and compared computationally
to a general purpose IP solver
in \S\ref{sec:computational_experiments}.

\subsection{IP Formulation}\label{sec:ip_formulation}

The optimization problem of our interest reads
\begin{gather}\label{eq:ip}
\begin{aligned}
    \min_{d}\quad & \sum_{i=1}^N c_i d_i + \alpha  \sum_{i=1}^{N-1}  \vert x_{i+1} + d_{i+1} - x_i - d_i \vert \eqqcolon C(d)\\
    \text{s.t.}\quad
    & x_i + d_i \in \Xi \text{ for all } i \in [N],\\
    & \sum_{i=1}^N \gamma_i \vert  d_i \vert \le \Delta,
\end{aligned}
\tag{TR-IP}
\end{gather}
where we have dropped a constant term from the objective.  It arises from \eqref{eq:tr} by discretizing of $(0,T)$ into $N \in \mathbb{N}$ intervals, see \cite{leyffer2021sequential}. The absolute values can be replaced with linear inequality constraints by introducing auxiliary variables, which yields an IP formulation.
Again, $\Xi = \{\xi_1,\ldots,\xi_m\} \subset \Z$, with $\xi_1 < \ldots < \xi_m$, $m \in \N$, 
denotes a finite set of integers.
Because the control functions are $\Xi$-valued, we use the ansatz of 
interval-wise constant functions
that are represented by the vector of step heights $x \in \Xi^N$.
The latter is an input of \eqref{eq:ip} and contains the 
step heights of the previously accepted control function iterate in 
\cref{alg:slip}.
The vector $\gamma \in \N^N$ contains the lengths of the discretization intervals. For uniform
discretization grids this means that $\gamma_i = 1$ for all $i \in [N]$.

The variable $d \in (\Xi - x_1) \times \ldots \times (\Xi - x_N)$ denotes the optimal step given
by the solution of \eqref{eq:ip}, yielding the new discretized control function representative $x + d \in \Xi^N$.
The term $\alpha  \sum_{i=1}^{N-1}  \vert x_{i+1} + d_{i+1} - x_i - d_i \vert$ in the objective models the sum of the jump heights between the subsequent intervals
scaled by a penalty parameter $\alpha > 0$, which is also an input of \eqref{eq:ip}.

Finally, the vector $c \in \R^N$ is a further input of \eqref{eq:ip} and arises from a numerical
approximation of $\nabla F(x)$ in \eqref{eq:tr}. We will frequently
use this setting for the quantities in \eqref{eq:ip} and thus
summarize it in the assumption below.

\begin{assumption}\label{assump:1}
Let $c \in \mathbb{R}^N$, $\alpha \in \mathbb{R}_{\geq 0}$, $\Delta \in \mathbb{N}$, $\Xi=\{\xi_1,\ldots,\xi_m\} \subset \mathbb{Z}$ with $\xi_1<\ldots<\xi_m$ for some $m \in \mathbb{N}$, $x \in \Xi^N$, and $\gamma \in \mathbb{N}^N$ be given.
\end{assumption}

\begin{remark}
Our restriction to $\gamma_i \in \N$ means that the different mesh sizes 
have to be integer multiples of the smallest mesh size. While this
does restrict the possible discretization grids severely, it enables
our algorithmic approach with shortest path algorithms below, which does not
generalize to arbitrary choices $\gamma_i \in (0,\infty)$.
\end{remark}

\begin{remark}
The inputs $x \in V^N$, $c \in \R^N$, and $\gamma \in \N^N$ may change in every outer iteration
of \cref{alg:slip} (the latter only in the presence of an adaptive grid refinement
and coarsening strategy) but are constant over the inner iterations of an outer iteration.
The trust-region radius $\Delta > 0$ changes in every inner iteration.
\end{remark}

\begin{remark}
We restrict the input $\Delta \in \N$ in our analysis to
\begin{gather}\label{eq:trivial_Delta_bound}
\Delta \le \Delta_{\max} \coloneqq (\max \Xi - \min \Xi) \|h\|_\infty N
\end{gather}
because the constraint $\sum_{i=1}^N \gamma_i \vert d_i \vert \le \Delta$ may be dropped in \eqref{eq:ip}
if $\Delta \ge \Delta_{\max}$.
\end{remark}
\subsection{NP-hardness of \eqref{eq:ip}}
\label{sec:np-hardness}
The problem \eqref{eq:ip} resembles the Knapsack problem closely. If we set $\alpha=0$, $x=(0,\ldots,0)$ and $\Xi=\{0,1\}$, the problem is reduced to 
\begin{gather*}
\begin{aligned}
    \min_{d}\quad & \sum_{i=1}^N c_i d_i \\
    \text{s.t.}\quad
    & d_i \in \{0,1\} \text{ for all } i \in [N] \text{ and } 
    \sum_{i=1}^N \gamma_i \vert  d_i \vert \le \Delta,
\end{aligned}
\end{gather*}
which is the well-known Knapsack problem. Obviously, the more general problem 
\eqref{eq:ip} is also NP-hard. We briefly show that even in the case of a
uniform discretization grid, meaning $\gamma_i = 1$ for all $i \in [N]$, 
\eqref{eq:ip} remains NP-hard. 

\begin{proposition}
Let \cref{assump:1} hold. The problem \eqref{eq:ip} with $\gamma_i=1$ for all $i \in [N]$ is NP-hard.
\end{proposition}
\begin{proof}{Proof.}
We show the NP-hardness by a reduction from the Knapsack problem.
Let $\tilde c_1$, $\ldots$, $\tilde c_n$ be the positive costs and
$\tilde w_1$, $\ldots$, $\tilde w_n$ the positive integer weights of the $n$ 
items. Let $K$ be the budget of the Knapsack.

We set $\Delta=K$ and $N=2n+1$. Let $\alpha > 0$ be an arbitrary but fixed.
We construct $\Xi$ such that it contains $N+1$ elements of the form
\begin{align*}
&\sum_{i=1}^k \tilde w_i+(k+1)(\Delta+1) \in \Xi \text{ for all } k \in \{0,\ldots,n-1\} \\ \text{ and } 
&\sum_{i=1}^k \tilde w_i+k(\Delta+1) \in \Xi \text{ for all } k \in \{0,\ldots,n\}.
\end{align*}
In the case that $\Xi$ is ordered from smallest to largest, the difference between two subsequent elements is either $\Delta+1$ or the weight of an item.
For all $i \in [N]$ we choose
\[
x_i = \begin{cases}
\sum_{j=1}^{\frac{i}{2}-1} \tilde w_j+\frac{i}{2}(\Delta+1) &\text{$i$ even},\\
0 &\text{$i$ odd}.
\end{cases}
\]
For odd $i$ this directly implies $d_i=0$, while for even $i$ it follows that $d_i=0$ or $d_i =\tilde w_{\frac{i}{2}}$.
Furthermore, we set
\begin{equation*}
c_i = \begin{cases}
-\frac{\tilde c_{\frac{i}{2}}}
{\tilde w_{\frac{i}{2}}}-2\alpha &\text{$i$ even},\\
0 &\text{$i$ odd}.
\end{cases}
\end{equation*}
This leads to the reduced problem
\begin{align*}
\begin{aligned}
    \min_{d}\ & \sum_{i=1}^{n}  c_{2i} d_{2i} + 2\alpha  \sum_{i=1}^{n} (x_{2i} + d_{2i}) \\
    \text{s.t.}\ & x_i + d_i \in V \text{ for all } i \in \{1,\ldots,N\}\\
            &
              \sum_{i=1}^n \vert d_{2i} \vert \le \Delta .\\
\end{aligned}
\end{align*}
It follows from the choice of the $c_i$ that
$c_{2i}d_{2i}+2\alpha d_{2i} = 
(-\frac{\tilde c_{i}}
{\tilde w_{i}}-2\alpha)
d_{2i}+2\alpha d_{2i} 
= -
\frac{\tilde c_{i}}
{\tilde w_{i}} d_{2i}$.
 Because either $d_{2i}=0$ or $d_{2i}=\tilde w_i$ holds, a solution of \eqref{eq:ip} corresponds to a solution of the Knapsack problem. 
\end{proof}

\begin{remark}
We note that the NP-hardness of the problem \eqref{eq:ip} is shown
by constructing a complicated set $\Xi$. From the application point
of view in integer optimal control, the set $\Xi$ is generally
a small set of integers and remains constant for all instances
of \eqref{eq:ip} that are generated during a run of \cref{alg:bin}.
Thus we construct efficient combinatorial algorithms, where
$\Xi$ is treated as a fixed input parameter in the remainder.
\end{remark}

\subsection{Reformulation of \eqref{eq:ip} as a Shortest Path Problem and Graph Construction}\label{sec:spp_reformulation}
The shortest path problem (SPP) in a graph $G(V,A)$ with weight function $w:A \to \mathbb{R}$ is the problem of determining the minimum weight path $p=\{s,v_1,\ldots, v_N,t\}$ between two nodes $s,t \in V$. The weight of $p$ is defined as $W(p)\coloneqq w_{s,v_1}+\sum_{i=1}^{N-1} w_{v_i,v_{i+1}} +w_{v_N,t}$.
To obtain the equivalent SPP formulation of \eqref{eq:ip}, we introduce
the parameterized family of IPs
\begin{gather}\label{eq:tr-ip-j}
\begin{aligned}
    \min_{\mathclap{d_{j+1},\ldots, d_N}}\quad & \sum_{i=j+1}^N c_i d_i + \alpha  \sum_{i=\max\{j,1\}}^{N-1}  \vert x_{i+1} + d_{i+1} - x_i - d_i \vert\\
    \text{s.t.}\quad
    & x_i + d_i \in \Xi \text{ for all } i \in \{j+1,\dots,N \},\\
    & \sum_{i=j+1}^N \gamma_i \vert  d_i \vert \le \Delta- \sum_{i=1}^j \gamma_i \vert d_i \vert. 
\end{aligned}
\tag{TR-IP($j$)}
\end{gather}
for all $j \in \{0,\ldots,N\}$. It is immediate that \tripjref{0} is an equivalent representation of \eqref{eq:ip}. For $j \in [N]$, the problem \tripjref{$j$}
corresponds to \eqref{eq:ip} with $d_1,\ldots,d_j$ being fixed. Specifically,
the \emph{resource capacity constraint}
$\sum_{i=1}^N \gamma_i\vert d_i\vert \le \Delta$ in \eqref{eq:ip} changes to
\[ \sum_{i=j+1}^N \gamma_i\vert d_i\vert
\le r(\Delta,d_1,\ldots,d_j) \coloneqq \Delta - \sum_{i=1}^j \gamma_i \vert d_i \vert \in \mathbb{Z}, \]
where we call $r(\Delta,d_1,\ldots,d_j)$ the \emph{remaining capacity} for
the resource capacity constraint in \eqref{eq:tr-ip-j}.

Because the problem \eqref{eq:tr-ip-j} cannot admit a feasible point if
$r(\Delta,d_1,\ldots,d_j) < 0$, we say that a \emph{remaining capacity is feasible}
if $r(\Delta,d_1,\ldots,d_j) \ge 0$. Because $\gamma_i \vert d_i \vert \in \mathbb{N}$
for all $i \in [N]$, a feasible capacity can only assume the integer values
between $0$ and $\Delta$.

The structure of the problem, specifically the shrinking feasible set
of \eqref{eq:tr-ip-j} for increasing $j$, allows us to construct a digraph
$G(V, A)$, where the set of nodes $V$ is partitioned into $N$ layers and
the directed edges (in the set $A$) can only exist between subsequent
layers, that is from layer $j$ to $j + 1$ for $j \in [N - 1]$.
\paragraph{Nodes in $G(V,A)$.}
Let $i \in [N]$, then the nodes in the layer $i$ encode 
feasibility of the $d_i$ with respect to the integrality constraint $x_i + d_i \in \Xi$ and the resource capacity 
constraint. Formally, a node $v \in V$ is a triplet
$v = (j, \delta, \eta) \in [N] \times \{x - y\,|\,x, y \in \Xi\}
\times \{0,\ldots,\Delta\}$. For $i \in [N]$, the layer $L_i$
is defined as the set of triplets
\[ L_i \coloneqq \Big\{ (i,\delta,\eta)\,\Big|\,
   \delta \in \Xi - x_i \text{ and }
   \eta \in \{\Delta - |\delta| \gamma_i,\Delta - |\delta| \gamma_i - 1,\ldots,0\}
   \Big\}
\]
and the set of nodes is
$V = L_1\;\dot{\cup}\;\cdots\;\dot{\cup}\;L_N$, 
where $\dot{\cup}$ denotes the disjoint union.
To access the entries of a node $v = (j,\delta,\eta) \in V$,
we define the notation
\[ \ell(v) = j,\enskip
   \tilde{d}(v) = \delta,\enskip\text{and}\enskip
   \tilde{r}(v) = \eta.
\]

\paragraph{Directed edges in $G(V,A)$.}
Then the set $A \subset V \times V$ of directed edges is defined
as
\[ (u, v) \in A\quad:\Longleftrightarrow\quad
\left\{
\begin{aligned}
    \ell(v) = \ell(u) + 1,&\\
    \text{there exists } (a,b) \in A \text{ with } b = u
    \text{ if } \ell(u) > 1,&\text{ and}\\
    \tilde{r}(v) = \tilde{r}(u)
                     - \gamma_{\ell(v)}\vert\tilde d(v)\vert.& \\
\end{aligned}
\right.
\]

The first condition guarantees that we obtain an LDAG
because edges only exist between subsequent layers, while the
second and third condition ensure that the resource capacity
constraint is satisfied inductively.
The weight of an edge $e=(u,v)$ is given by
\[w_{(u,v)} = c_{\ell(v)}  \tilde d(v)+\alpha \vert x_{\ell(v)} - x_{\ell(u)} +  \tilde d(v) -  \tilde d(u) \vert.\]
Furthermore a source $s$ and a sink $t$ are added to the graph. The 
source $s = (0,\emptyset,\Delta)$
is connected to all $v \in V$ in the first layer
with sufficient remaining capacity, that is
\[ (s,v) \in A\quad:\Longleftrightarrow\quad
   \ell(v) = 1 \text{ and }\tilde{r}(v) = \Delta - |\tilde{d}(v)| \gamma_1.
\]
Moreover, we have the weight $w_{(s,v)} = c_1  \tilde{d}(v)$.
The sink $t = (N + 1,\emptyset,0)$ is connected to each node
$n \in V$ in the last layer that has an incoming edge, that is
\[ (v,t) \in A\quad:\Longleftrightarrow\quad
   \text{there exists } u \in V \text{ such that } (u,v) \in A. \]
The weights have the value zero, that is $w_{(v,t)} = 0$.

\Cref{fig:graph_construction} shows a sketch of such
a graph $G(V,A)$. It depicts the layered structure as well as the
directed edges encoding feasible choices from one layer to the next.

The construction detailed above implies immediately the following upper bounds on the cardinalities of
$V$ and $A$ for the graph $G$:
\begin{gather}\label{eq:upper_bound_V}
\vert V \vert \leq N \cdot (\Delta+1) \cdot \vert \Xi \vert+2
\end{gather}
and
\begin{gather}\label{eq:upper_bound_A}
\vert A \vert \leq \vert \Xi \vert ^2 \cdot N \cdot (\Delta+1)+\vert \Xi \vert+(\Delta +1) \cdot \vert \Xi \vert.    
\end{gather}

\begin{figure}
    \centering
	\resizebox{1.0 \textwidth}{!}{%
	
	\begin{tikzpicture}[node distance={65mm}, thick, main/.style = {draw, circle}] 
	\node[main, minimum size = 2.85 cm] (0) {\Huge $s$};
	\node[minimum size =2.25 cm] (-1) [right of=0] {};

	\node[main, minimum size = 2.25 cm] (4) [right of=-1] {\Huge $(1,0,2)$};
	\node[main,minimum size =2.25 cm] (5) [below of = 4] {\Huge $(1,0,1)$}; 
	\node[main, minimum size = 2.25 cm] (6) [below of=5] {\Huge$(1,0,0)$};
	\node[main, minimum size = 2.25 cm] (7) [right of=5] {\Huge$(1,1,1)$}; 
	\node[main, minimum size = 2.25 cm] (8) [above of=7] {\Huge$(1,1,2)$}; 
	\node[main, minimum size = 2.25 cm] (9) [below of=7] {\Huge$(1,1,0)$};
	\node[draw,dotted,fit=(5) (4) (6) (7) (8) (9)] (100){};
	
	\node[] (101) [left of= 5] {\Huge {ordered by capacity $\downarrow$} };
	\node[] (105) [above of = 100] {};
	\node[] (103) [above of= 105, yshift = -3.5cm] {\Huge {ordered by corresponding $x_i+d_i=\xi \in \Xi$ $\rightarrow$}};
	
	\node[minimum size =1.25 cm] (-2) [right of=7] {};

	\node[main, minimum size = 2.25 cm] (14) [right of=-2] {\Huge$(2,0,1)$}; 
	\node[main, minimum size = 2.25 cm] (13) [above of=14] {\Huge$(2,0,2)$}; 
	\node[main, minimum size = 2.25 cm] (15) [below of=14] {\Huge$(2,0,0)$};
	\node[main, minimum size = 2.25 cm] (16) [right of=14] {\Huge$(2,1,1)$}; 
	\node[main, minimum size = 2.25 cm] (17) [above of=16] {\Huge$(2,1,2)$}; 
	\node[main, minimum size = 2.25 cm] (18) [below of=16] {\Huge$(2,1,0)$};
	\node[draw,dotted,fit=(13) (14) (16) (17) (18) (15)] {};
	\node[minimum size =2.85 cm] (-3) [right of=16] {\Huge $\dots$};
    \node[main, minimum size = 2.85 cm] [right of=-3] (20) {\Huge $t$};
	
	\draw[->, dotted] (0)  -- node[midway,xshift=-0.5cm, above left,sloped] {\Huge $0$} (4);
    \draw[->, dotted] (0) -- node[near start, above,sloped] {\Huge $c_1$} (7);
    
    \draw[->,dotted] (4)  -- node[midway,xshift=3.7cm, above left,sloped] {\Huge $0$} (13);
    \draw[->,dotted] (4) -- node[midway, above,sloped] {\Huge $c_2+\alpha$} (16);
    
    \draw[->,dotted] (7)  -- node[midway,xshift=0.5cm, above left,sloped] {\Huge $0$} (14);
    \draw[->,dotted] (7) -- node[midway,xshift=-3.5cm, above,sloped] {\Huge $c_2+\alpha$} (18);

	\end{tikzpicture}
} 
\caption{Example of the graph construction for arbitrary $c \in \mathbb{R}$ and $n \in \mathbb{N}_{\geq 2}$,  $\Xi=\{0,1\}$, $x_1,x_2=0$, $\Delta=2$. Edges always point from left to right as they only connect subsequent layers. Because the remaining capacity is decreasing monotonously, this introduces a level structure where each level represents a remaining capacity between $0$ and $\Delta$. Edges can only point to nodes of the same or a lower level.}\label{fig:graph_construction}
\end{figure}
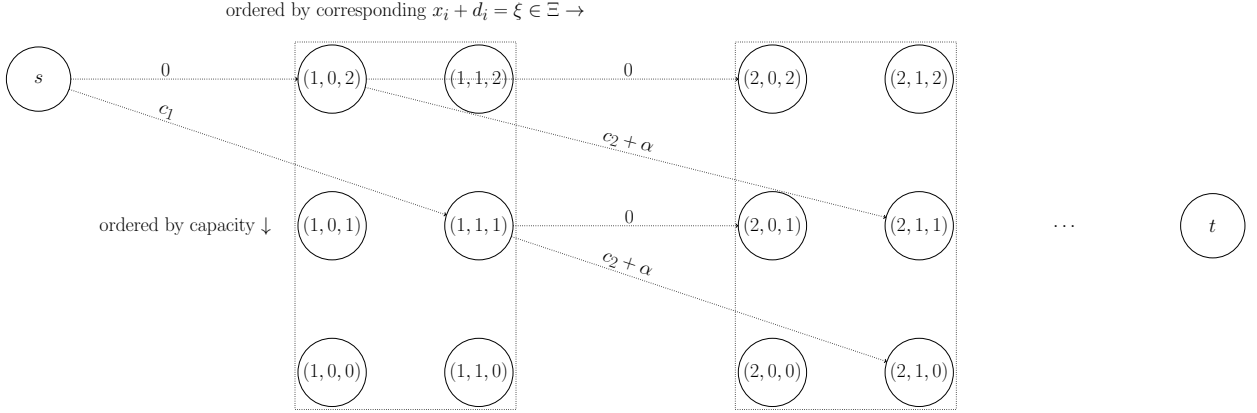

\begin{proposition}\label{prp:sol_of_ip_is_sol_of_spp}
Let \cref{assump:1} hold. Then there is a one-to-one correspondence between solutions of \eqref{eq:ip}
and solutions of the SPP on $G$ from $s$ to $t$. 
\label{prop:SppIP}
\end{proposition}
\begin{proof}{Proof.}
Let $p=\{s,v_1,\ldots,v_N,t\}$ be a path in $G$. 
The remaining capacity in the node $n_N$ is given by
\[ 0 \le \tilde r(v_N)  =\tilde r(v_{N-1})- \gamma_N \vert  \tilde d(v_N) \vert =\ldots = \tilde r(v_1)-\sum_{i=2}^N \gamma_i \vert  \tilde d(v_i) \vert =  \Delta - \sum_{i=1}^{N}\gamma_i \vert  \tilde d(v_i) \vert.
\]
It follows that the capacity constraint holds for the point
\[ d=(d_1,\ldots,d_N) = (\tilde d(v_1),\ldots,\tilde d(v_N)).
\]
Moreover $x_i+d_i \in \Xi$ holds by construction and thus the point $d$ is 
feasible. Moreover, the cost $C(d)$ of the feasible point $d$ is equal to the weight $W(p)$ of the path $p$, which follows from
\begin{align*}
C(d)
&= \sum_{i=1}^{N}c_i  d_i+ \alpha \sum_{i=2}^{N} \vert x_{i} - x_{i-1} +   d_i -   d_{i-1} \vert \\
&= c_1  \tilde d(v_1)+\sum_{i=2}^{N}(c_i  \tilde d(v_i)+\alpha  \vert x_{i} - x_{i-1} +  \tilde d(v_i) -  \tilde d(v_{i-1}) \vert) = w_{s,v_1} + \sum_{i=2}^{N} w_{v_{i-1},v_i} =W(p).
\end{align*}

Let $d=(d_1,\ldots,d_N)$ be feasible with cost $C(d)$.
This results in the path $p=\{s,v_1,\ldots,v_N,t\}$ with
\[ 
v_i=(\ell(v_i), \tilde d(v_i), \tilde r(v_i))=\left(i, d_i,\Delta-\sum_{j=1}^i \gamma_j \vert d_j \vert\right)
\]
where $\tilde d(v_i)+x_i=d_i+x_i \in \Xi$ for $i \in [N]$ by construction. Exactly as above, the weight $W(p)$ of the path equals the cost $C(d)$. 
Consequently, every feasible point coincides with a path from $s$ to $t$ in $G$. 
\end{proof}

\subsection{Reformulation of the SPP on $G$ as a Resource Constrained Shortest Path Problem}\label{sec:rcspp_formulation}
The resource constrained shortest path problem (RCSPP) on a graph $G_c(V_c,A_c)$ is the problem of determining the shortest path which adheres to a resource constraint. 
To obtain an RCSPP on a quotient graph, we define an equivalence relation $\sim_G$ on the set of nodes. Let $u,v \in V$. Then we can define
\begin{align*}
    u \sim_G v \hspace{0.5em} :\iff \hspace{0.5em} \ell(u) = \ell(v) \quad \text{and} \quad \tilde d(u) =\tilde d(v).
\end{align*}
This means that nodes in the same equivalence class only differ in their remaining capacity. We denote the equivalence class containing $v \in V$ as $[v]$. This naturally leads to a reformulation of the SPP on $G$ as an RCSPP on the quotient graph. Equivalence classes are nodes in the new LDAG $G_c( V_c,  A_c)$ with fully connected subsequent layers. Each layer $i$ contains a node for each feasible choice of $d_i$. Thus a node is a pair of the layer $i$ and the choice of $d_i$. Again, a source $s$ and a sink $t$ are added to the graph $G_c$ and connected to each node of the first and last layer respectively. Since the weights of the edges in $G$ do not depend on the remaining capacity of the incident nodes, the weights are defined exactly as for $G$.
Additionally, we assign a resource consumption to each edge. An edge from the source $s$ to a node $u$ in the first layer has a resource consumption of $ \gamma_{1} \vert \tilde d(u) \vert$, while an edge from a node $u$ to a node $v$ has a resource consumption of $ \gamma_{\ell(v)} \vert \tilde d(v) \vert$. Edges into the sink $t$ have a resource consumption of $0$.
\begin{proposition}\label{prp:sol_of_spp_is_sol_of_rcspp}
Let \cref{assump:1} hold.
Then there is a one-to-one correspondence between solutions of the RCSPP on $G_c$ and solutions of the SPP on $G$.
\end{proposition}
\begin{proof}{Proof.}
We show that the solution of the RCSPP coincides with the solution of the IP.
Let $p=\{s,v_1,\ldots,v_N,t\}$ be a path in $G_c$. 
Because the path is feasible for the RCSPP,
\[ \sum_{i=1}^N  \gamma_i \vert \tilde d(v_i) \vert \leq \Delta \]
follows.
Thus the point 
\[d=(d_1,\ldots,d_N)=( \tilde d(v_1),\ldots, \tilde d(v_N))\]  
is feasible for the IP. The cost $C(d)$ of the feasible point $d$ matches the weight $W(p)$ of the path as shown in \cref{prop:SppIP}.

Let  $ d=(d_1,\ldots,d_N)$ be a feasible point with cost $C(d)$. This leads to the path $p=\{s,v_1,\ldots,v_N,t\}$ with \[v_i=(\ell(v_i),\tilde  d(v_i))=(i,d_i).\] The weight $W(p)$ of the path $p$ equals the cost $C(d)$ of the point $d$ as above.
Because of the feasibility of $d$, the inequality 
$ \sum_{i=1}^N  \gamma_i \vert \tilde d(v_i) \vert = \sum_{i=1}^N  \gamma_i \vert d_i \vert  \leq \Delta $
guarantees that the path is feasible. A shortest constrained path on $G_c$ coincides with a minimal feasible point of the IP and in turn is equivalent to the shortest path on $G$. 
\end{proof}
\begin{remark}
Because $c_id_i<0$ is possible the weights of the edges may also be negative. To ensure non-negative weights a uniform offset can be added to all weights. Because all paths from $s$ to $t$ have the same length, the cost of all paths is increased by the same fixed amount.
\end{remark}

\subsection{Solution Algorithms}\label{sec:solution algo}
The problem \eqref{eq:ip} can be reduced to an SPP on an LDAG (see \cref{prop:SppIP}), which yields a topological order
of the nodes. The topological order implies that the shortest path from $s$ to $t$ can be found in
$\mathcal{O}(|V| + |A|)= \mathcal{O}(N \cdot \Delta \cdot \vert \Xi \vert^2)$, see \cite[page 655 ff.]{cormen2009algorithm}, 
which yields the existence of a pseudo-polynomial algorithm for \eqref{eq:ip}.
This solution approach provides a better worst case complexity than solving the SPP formulation with Dijkstra's algorithm 
because it avoids the need of a priority queue. It cannot utilize the underlying structure to terminate early, however. 
We set forth to augment Dijkstra's algorithm so that it is able to utilize the problem structure and eventually
outperform the topological sorting-based approach in practice. The derived accelerations transform Dijkstra's
algorithm into an $A^*$ algorithm, which arises from Dijkstra's algorithm by adding an additional heuristic function
$h: V \to \mathbb{R}$. For each node $v$, the evaluation $h(v)$ estimates for the cost to reach the sink $t$ from $v$.
Instead of expanding the node with the lowest path cost, the $A^*$ algorithm expands the node with the lowest sum of
the path cost and the heuristic cost estimation. Ties are broken arbitrarily.

While $A^*$ might be a mere heuristic without further assumptions on $h$, a careful choice of $h$ allows to
conserve the optimality of the solution guaranteed by Dijkstra's algorithm. We provide the corresponding
concept of a \emph{consistent heuristic} and reference the resulting optimality guarantee below.
\begin{definition}
A heuristic function $h:V \to \mathbb{R}$ is called \emph{consistent} if $h(t) = 0$ holds
for the sink $t$ and the triangle inequality is satisfied for every edge, specifically
\[ w_{u,v}+h(v) - h(u) \ge 0 \text{ for all } e=(u,v) \in A. \]
\end{definition}
\begin{proposition}
Let $G(V,A)$ be a graph with source $s$ and sink $t$. The $A^*$ algorithm determines a shortest $s$-$t$ path
if the heuristic is consistent.
\end{proposition}
\begin{proof}{Proof.}
We refer the reader to \cite[page 82\,ff.]{pearl1984}.\end{proof}
If the heuristic function is consistent, the $A^*$ algorithm coincides with Dijkstra's algorithm with the weights
$\bar w_{u,v}=w_{u,v}+h(v) - h(u)$ and thus solves the SPP formulation in
$\mathcal{O}(\vert  V \vert \log(\vert V \vert) + \vert A \vert)$ (\cite{korte2018} page 162).
This worst case complexity is again higher than the aforementioned topological sorting-based approach
described in \cite[page 655\,ff.]{cormen2009algorithm}. The accelerations derived in \S\ref{sec:lagrangian}
below pay off in our computational experiments in 
\S\ref{sec:computational_experiments}-\ref{sec:numresults}
and the $A^*$ algorithm
shows superior performance particularly on larger instances of \eqref{eq:ip}
(where larger shall be understood with respect to the values of $N$ and $\Delta$).

\section{Acceleration of $A^*$ using the Lagrangian Relaxation}\label{sec:lagrangian}
The reformulation as an RCSPP allows us to use the Lagrangian relaxation of 
\eqref{eq:ip} to derive lower and upper bounds for
the RCSPP on the costs of the optimal paths in $G_c$ by following the ideas
from \cite{dumitrescu2003}, which are also valid for the SPP on $G$.
Furthermore, we derive a consistent heuristic function for the $A^*$ algorithm 
from the Lagrangian relaxation. We provide the Lagrangian relaxation in 
\S\ref{sec:lagrangian_relaxation} with respect to the RCSPP formulation
derived in \S\ref{sec:rcspp_formulation}. This Lagrangian relaxation is
used in \S\ref{sec:heuristic} to derive a consistent heuristic
function for the $A^*$ algorithm that operates on the SPP reformulation of \eqref{eq:ip}. We show how an optimal Lagrange multiplier for the
Lagrangian relaxation may be determined in a preprocessing
stage in \S\ref{sec:preprocessing} and provide a dominance principle
to further reduce the search space in \S\ref{sec:dominated}.

\subsection{Lagrangian Relaxation}\label{sec:lagrangian_relaxation}
The authors in \cite{dumitrescu2003} consider the RCSPP in the
generalized form
\begin{gather*}
	\text{z}(s,t,\Delta) = \left\{
	\begin{aligned}
	\min \quad & W(P)\\
	\text{s.t.}\quad & P \in \textbf{P}^{s,t} \ \text{and} \ R(P) \leq \Delta,
	\end{aligned}
	\right. 
\end{gather*}
where $\textbf{P}^{s,t}$ is the set of all paths from $s$ to $t$. The cost of a path $P \in \textbf{P}^{s,t}$ is denoted by $W(P)$, while $R(P)$ is the resource consumption of the path. The Lagrangian relaxation of this formulation is 
\begin{gather*}
	\max_{\lambda \geq 0} \LR (s,t,\Delta,\lambda) \coloneqq \max_{\lambda \geq 0} \left\{
	\begin{aligned}
	\min \quad & W(P)+\lambda(R(P)-\Delta)\\
	\text{s.t.}\quad & P \in \textbf{P}^{s,t}.
	\end{aligned}
	\right. 
\end{gather*}
We call $\lambda^* \in \argmax_{\lambda\geq 0} \LR(s,t,\Delta,\lambda)$ an optimal Lagrange multiplier. If the set of feasible paths for the RCSPP is not empty, the set $\argmax_{\lambda\geq 0} \LR(s,t,\Delta,\lambda)$ is not empty and contains either a single value or a single interval, because the Lagrangian function is concave with respect to $\lambda$. (see \cite{xiao2005}) 
Obviously, the inequality $z(s,t,\Delta) \geq \LR(s,t,\Delta,\lambda)$ holds for all $\lambda \geq 0$. This approach can be extended to subpaths from a node $v$ with a resource consumption $\mu$, which leads to the problem formulation
	\begin{gather*}
	\text{z}(v,t,\Delta-\mu) = \left\{
	\begin{aligned}
	\min \quad & W(P)\\
	\text{s.t.}\quad & P \in \textbf{P}^{v,t} \ \text{and} \ R(P) \leq \Delta- \mu,
	\end{aligned}
	\right.
	\end{gather*}
and its relaxation
	\begin{gather*}
	\LR(v,t,\Delta-\mu,\lambda) = \left\{
	\begin{aligned}
	\min \quad & W(P)+\lambda(R(P)-(\Delta-\mu))\\
	\text{s.t.}\quad & P \in \textbf{P}^{v,t}.
	\end{aligned}
	\right.
	\end{gather*}
The relaxation can be written as
\begin{gather*}
    \LR(v,t,\Delta-\mu,\lambda)= -\lambda(\Delta-\mu)+\zeta(v,t,\lambda) \quad \text{with} \quad 
	\zeta(v,t,\lambda) = \left\{
	\begin{aligned}
	\min \quad & W(P)+\lambda R(P)\\
	\text{s.t.}\quad & P \in \textbf{P}^{v,t}.
	\end{aligned}
	\right.
\end{gather*}
The term $\zeta$ can be calculated as the solution of the shortest path problem on $G_c^\lambda$, which is the graph we obtain by increasing each edge weight of the graph $G_c$ by the product of the Lagrange multiplier $\lambda$ and the resource consumption of the edge.
For all $\lambda \geq 0$ the value of $\LR(v,t,\Delta-\mu,\lambda)$ provides a lower bound for the cost of the remaining path from $v$ to $t$. \cite{dumitrescu2003} use a cutting plane approach to solve the Lagrangian relaxation for the RCSPP. In the process, they obtain a sequence of Lagrange multipliers $\lambda$ for which the corresponding $\zeta$ are calculated. We refer to this set of all calculated $\lambda$ as $\Lambda$.
For a node $v$ and a remaining capacity $\mu$ the term $\LR(v,t,\mu,\lambda)=-\lambda \mu+\zeta(v,t,\lambda)$ is a lower bound for the shortest constrained path from $v$ to $t$ in $G_c$ for an arbitrary $\lambda \geq 0$.  \cite{dumitrescu2003} use $\max_{\lambda \in \Lambda} -\lambda \mu+\zeta(v,t,\lambda)$ to construct lower bounds for each node $v$. If an $s$-$v$ path with a remaining capacity of $\mu$ is found, the path is cut if the sum of the cost of the $s$-$v$ path and the lower bound for the $v$-$t$ path exceeds a known valid upper bound. Any feasible path in $G_c$ yields an upper bound.  Additionally, any subpath in $G_c^\lambda$ from a node $v$ to $t$ obtained during the preprocessing stage can be used to construct an upper bound as the sum of the costs of the shortest $s$-$v$ path and the subpath if the sum of the capacity consumption for the shortest $s$-$v$ path and the subpath is not greater than $\Delta$. 

\subsection{Heuristic Function}\label{sec:heuristic}
The Lagrangian relaxations can be used to obtain not just upper and lower bounds but a consistent heuristic function for the $A^*$ algorithm.

By construction of $G_c$, an equivalence class of vertices of $G$ is a node in $G_c$. This allows us to extend the
Lagrangian relaxations for the nodes of $G_c$ to the nodes of $G$. We recall that a node $v \in  V$ satisfies $v \in [u]$
for $[u] \in V_c$
if
\[ \tilde d(v)=  \tilde d(u) \quad \text{and} \quad \ell(v)=\ell(u).\]

\begin{theorem}
Let \cref{assump:1} hold. Then for all $\lambda \geq 0$ the heuristic function
\[h_\lambda(v)=\LR([v],t,\tilde r(v),\lambda)=-\lambda \tilde r(v)+\zeta([v],t,\lambda)\]
is consistent for the $A^*$ algorithm on $G$.
\label{theo:heuristic}
\end{theorem}
\begin{proof}{Proof.}
The remaining capacity in $t$ is $0$ and $\zeta(t,t,0)=0$, implying that $h_\lambda(t)=0$.\\
Let $e=(u,v)$ be an edge in $G$.
To ensure that the heuristic function is consistent we have to show that \[h_\lambda(u) \leq w(u,v)+h_\lambda(v).\]
Inserting the definition of $h_\lambda$ leads to the equivalence
\begin{align*}
  h_\lambda(u) \leq w(u,v)+h_\lambda(v)
\iff    \zeta([u],t,\lambda) - \zeta([v],t,\lambda) \leq w(u,v) + \lambda \vert \tilde d(v)\vert .
\end{align*}
For every $\lambda \geq 0$, the costs of the shortest paths from $[u]$ and $[v]$ to $t$ in the graph $G_c^\lambda$ are $\zeta([u],t,\lambda)$ and $\zeta([v],t,\lambda)$. The weight of the edge from $[u]$ to $[v]$ in $G_c^\lambda$ is $w(u,v) + \lambda \vert \tilde  d(v)\vert$. Because $\zeta([u],t,\lambda)$ and $\zeta([v],t,\lambda)$ are the costs of the shortest paths from $[u]$ and $[v]$ to $t$ in $G_c^\lambda$ respectively, the weight of any edge from $[u]$ to $[v]$ is at least $\zeta([u],t,\lambda)- \zeta([v],t,\lambda)$ (otherwise this would be a contradiction to being the costs of the shortest paths). Using the above equivalences, the consistency of the heuristic function follows.
\end{proof}
Combining the Lagrangian relaxations for multiple values of $\lambda$ has been shown to improve the bounds
(see \cite{dumitrescu2003}). Therefore, we also combine the information of heuristic functions $h_\lambda$
for different values of $\lambda$ into one.
\begin{proposition}\label{prp:final_heuristic}
Let \cref{assump:1} hold. Let $\Lambda \subset [0,\infty)$ with $\vert \Lambda \vert < \infty$ be given.
Then the heuristic function
\begin{align*}
    h_{\LR}(v)=\max_{\lambda \in \Lambda} \LR([v],t,\tilde r(v),\lambda)
\end{align*}
is consistent on $G$.
\end{proposition}
\begin{proof}{Proof.}
Let $h_1, \ldots, h_n$ be consistent heuristics on $G$ and $e=(u,v)$ an edge in $G$. Let
$i \in \argmax_{j \in [n]} h_j(u)$. Then
\[ w(u,v) \ge h_i(u)-h_i(v)
          = \max_{j \in [n]} h_j(u) - h_i(v)
          \ge \max_{j \in [n]} h_j(u) - \max_{k \in [n]} h_k(v)
\]
holds, which proves that $h = \max_{j \in [n]} h_j$ is a consistent heuristic. Thus $h_{\LR}$ is consistent
because \cref{theo:heuristic} showed the consistency of the $h_\lambda$.
\end{proof}

\subsection{Preprocessing Stage}\label{sec:preprocessing}
We show that the calculation of the shortest paths in the relaxed graphs and the determination of an $\varepsilon$-optimal
Lagrange multiplier may take place in a preprocessing stage. To this end, we use the equivalence of feasible and optimal
points for \eqref{eq:ip} and feasible and shortest paths for the corresponding SPPs and RCSPPs as is argued
in \cref{prp:sol_of_ip_is_sol_of_spp,prp:sol_of_spp_is_sol_of_rcspp}.
We use a binary search with the initial upper bound
\[ u \coloneqq c_{\max}+2 \alpha \quad \text{with} \quad  c_{\max} \coloneqq \max_{i \in [N]} \vert c_i \vert\]
and the initial lower bound $\ell = 0$, which is given in \cref{alg:bin}.
\begin{algorithm}
	\caption{Binary Search}\label{alg:bin}
	\hspace*{\algorithmicindent}
	\textbf{Input:} $G_c(V_c,A_c)$, source $s$, sink $t$, weights $w: A_c \to \mathbb{R}_{\geq0}$, $u=\vert \vert c \vert \vert_\infty+2 \alpha$, $\ell=0$, $\varepsilon>0$\\
    \hspace*{\algorithmicindent} \textbf{Output:} all shortest paths to $t$ for all calculated values of $\lambda$ 
	\begin{algorithmic}[1]
	    \While{$u-\ell \geq \varepsilon$}
	        \State{$\lambda \leftarrow \frac{u-\ell}{2}$}
	        \State {$\tilde w(u,v) \leftarrow w(u,v) + \lambda |\tilde{d}(v)|$ for all edges $e=(u,v)$ (see \S\ref{sec:spp_reformulation}-\S\ref{sec:rcspp_formulation})	        }\label{ln:lagrangian_relaxation_added}
	        \State{$P_\lambda \leftarrow$ calculate the shortest paths from all nodes $v$ to $t$ for $G_c$ with the weights $\tilde w$}
	        \State{$p_\lambda \leftarrow$ shortest s-t path in $P_\lambda$
	        that minimizes the resource consumption}
	        \If{resource consumption of $p_\lambda$ from $s$ to $t$ exceeds $\Delta$}\label{ln:resource_consumption_check}
	            \State{$\ell \leftarrow \lambda$}
	        \Else
	            \If{resource consumption of $p_\lambda$ equals $\Delta$}
	                \State{\textbf{Terminate early, $p_\lambda$ is optimal for $G$.}}
	           \EndIf
	           \State{$u \leftarrow \lambda$}
	        \EndIf
	    \EndWhile
	\State{\textbf{return} shortest paths to $t$ and the corresponding $\lambda$}
	\end{algorithmic}
\end{algorithm}
In each iteration of \cref{alg:bin}, an \emph{all shortest paths problem} to $t$ is solved. Due to the LDAG structure,
the latter can be solved with the algorithm detailed in \cite[page 655 ff.]{cormen2009algorithm}. If several paths have the
same cost, the path with the lesser resource consumption is chosen. Remaining ties are broken arbitrarily. Because a
duality gap may prevail, none of the calculated shortest $s$-$t$ paths are necessarily optimal for the SPP on $G$. If,
however, the remaining capacity for one of the calculated shortest paths is $0$, the path is already optimal for $G$,
which we show below in \cref{prop:welldefined}. Afterwards, we show in \cref{prop:binsearch} that the binary search
is well-defined meaning that
it terminates after finitely many steps, and returns a Lagrange multiplier that is within an $\varepsilon$-distance of
an optimal one.

Because the remaining capacity of a path is given by the difference
between the sum of $|\tilde{d}(v)|$ over all nodes $v$ in the path
and the input $\Delta$, the Lagrangian relaxation of the resource
constraint can be integrated into the weight function by adding the term
$\lambda |\tilde{d}(v)|$ when an edge that points to the node $v$ is added 
to the graph. This is done in \cref{alg:bin}
\cref{ln:lagrangian_relaxation_added} and gives a one-to-one relation
to the Lagrangian relaxation term
$\lambda \left(\sum_{i=1}^N |d_i| - \Delta\right)$
in terms of the problem formulation \eqref{eq:ip}.
\begin{proposition}\label{prop:welldefined}
Let \cref{assump:1} hold.
Let $d^*$ be an optimal solution of the relaxed problem
\begin{gather}
  \begin{aligned}
    \min_{d}\ & \sum_{i=1}^N c_i d_i +
    \alpha  \sum_{i=1}^{N-1} |x_{i+1} + d_{i+1} - x_i - d_i|
    + \lambda \left(\sum_{i=1}^N \gamma_i |d_i| - \Delta\right)
    \eqqcolon C_\lambda(d)\\
    \emph{s.t.}\  & x_i + d_i \in \Xi \text{ for all } i \in [N]
\end{aligned}
\label{eq:relax}
\tag{$\dagger_\lambda$}
\end{gather}
for a fixed $\lambda \geq 0$. If $\sum_{i=1}^N \gamma_i |d^*_i| - \Delta = 0$,
then $d^*$ is optimal for \eqref{eq:ip}.
\label{theorem:optlagr}
\end{proposition}
\begin{proof}{Proof.}
Let $\lambda \geq 0$, then
every feasible point $d$ of \eqref{eq:ip} satisfies
$C_\lambda(d) \leq C(d)$.
Let $d^*$ be an optimal solution of \eqref{eq:relax} with
$\sum_{i=1}^N \gamma_i |d^*_i| - \Delta = 0$, then $d^*$ is also feasible for \eqref{eq:ip} and the optimality follows from $C(d^*)=C_\lambda(d^*) \leq C_\lambda(d) \leq C(d)$. 
\end{proof}
If such a vector $d$ as in the claim of \cref{prop:welldefined} / an optimal
path for the RCSPP reformulation as is found, then the binary search may 
terminate early and the $A^*$ algorithm or any other solution algorithm
may be skipped entirely.
In order to prove that the binary search terminates after finitely many
steps and returns an optimal Lagrange multiplier, we need two auxiliary
lemmas, which are stated and proven below.
\begin{lemma}\label{lemma:bin1}
Let \cref{assump:1} hold. Let $\lambda \geq c_{\max}+2\alpha$.
Then $d \equiv 0$ is an optimal solution of \eqref{eq:relax}.
\end{lemma}
\begin{proof}{Proof.}
Let $d^*$ be an optimal solution of \eqref{eq:relax}. For the cost of $d^*$
we observe
\[ C_\lambda(d^*) + \lambda \Delta
= \sum_{i=1}^N c_i d^*_i 
+ \alpha  \sum_{i=1}^{N-1}|x_{i+1} + d^*_{i+1} - x_i - d^*_i|
+ \lambda \sum_{i=1}^N \gamma_i |d^*_i|.
\]
The second term satisfies the inequalities
\begin{gather}\label{eq:bin1}
|x_{i+1}-x_i+d^*_{i+1}-d^*_i|
\ge \big||x_{i+1} - x_i| - |d_{i+1} - d_i|\big| 
\ge |x_{i+1} - x_i| - (|d_{i+1}| + |d_i|).
\end{gather}
The first term satisfies the inequalities
\begin{gather}\label{eq:bin2}
\sum_{i=1}^N c_i d_i^* + c_{\max} \sum_{i=1}^N |d_i^*| 
\ge - \sum_{i=1}^N |c_i d_i^*| +  \sum_{i=1}^N c_{\max} |d_i^*|
\ge \sum_{i=1}^N (c_{\max} - |c_i|)|d_i^*|
\ge 0.
\end{gather}
it follows that
\begin{align*} 
\hspace{4.5em}& \hspace{-4.5em}
\sum_{i=1}^N c_i d^*_i 
+ \alpha\sum_{i=1}^{N-1}|x_{i+1} + d^*_{i+1} - x_i - d^*_i|
+ \lambda \sum_{i=1}^N |d^*_i| \\
\stackrel{\mathclap{\eqref{eq:bin1}}}{\geq} &
\sum_{i=1}^N c_i d^*_i
+ \alpha  \sum_{i=1}^{N-1} \left(|x_{i+1} - x_i| - (|d^*_{i+1}| + |d^*_i|)\right)
+ (c_{\max} + 2\alpha) \sum_{i=1}^N \gamma_i |d^*_i| \\
\stackrel{\mathclap{\eqref{eq:bin2}}}{\geq} 
& \alpha  \sum_{i=1}^{N-1} |x_{i+1} - x_i|
=C_\lambda(0)+\lambda \Delta,
\end{align*}
where we have also used $\gamma \in \N^N$ for the second inequality.
Therefore, $d \equiv 0$ is optimal.
\end{proof}

\begin{lemma}\label{lemma:bin2}
Let \cref{assump:1} hold. Let $\lambda \ge 0$ be fixed. 
Let $d^*$ be a minimizer of $d \mapsto \sum_{i=1}^N \gamma_i |d_i|$
over the set of optimal solutions of \eqref{eq:relax}. Let
$\lambda^*$ be an optimal solution of
\begin{gather}\label{eq:opt_lagrange_mult}
\begin{aligned}
    \argmax_{\hat{\lambda} \geq 0} \min_{d}\ & \sum_{i=1}^N c_i d_i + \alpha  \sum_{i=1}^{N-1} \vert x_{i+1} + d_{i+1} - x_i - d_i \vert + \hat{\lambda} \left (\sum_{i=1}^N \gamma_i |d_i| - \Delta \right )\\
    \emph{s.t.}\  &x_i + d_i \in \Xi \text{ for all } i \in [N].
\end{aligned}
\end{gather}
\begin{enumerate}[label=(\roman*)]
\item\label{itm:strictly_above_lambda} If $\sum_{i=1}^N \gamma_i|d^*_i| - \Delta > 0$
holds, then the inequality $\lambda < \lambda^*$ holds.
\item\label{itm:less_or_equal_lambda} If $ \sum_{i=1}^N \gamma_i|d^*_i| - \Delta < 0$
holds, then the inequality $\lambda \ge \lambda^*$ holds.
\end{enumerate}
\end{lemma}
\begin{proof}{Proof.}
Let $0\leq \lambda_1< \lambda_2 <\lambda_3$. The feasible sets of $(\dagger_{\lambda_i})$, $i=1,2,3$, are identical because
the remaining constraint does not depend on $\lambda$. We also observe that any feasible point $d$ of \eqref{eq:relax}
satisfies
\begin{gather}
\left\{
	\begin{aligned}
	&C_{\lambda_1}(d)<C_{\lambda_2}(d)<C_{\lambda_3}(d) \text{ if }
	\sum_{i=1}^N \gamma_i \vert d_i \vert - \Delta> 0,\\
	&C_{\lambda_1}(d)=C_{\lambda_2}(d)=C_{\lambda_3}(d) \text{ if }
	\sum_{i=1}^N  \gamma_i \vert d_i \vert - \Delta = 0, \text{ and}\\
	&C_{\lambda_1}(d)>C_{\lambda_2}(d)>C_{\lambda_3}(d) \text{ if }
	\sum_{i=1}^N  \gamma_i \vert d_i \vert - \Delta< 0.
	\end{aligned}
	\right. 
\label{ineqlam}
\end{gather}
We proceed with $d^*$ as assumed with the choice $\lambda_2 \coloneqq \lambda$,
that is $d^*$ is an optimal solution of
\begin{gather}\label{eq:minimize_sum_norm_over_solution_set}
\begin{aligned}
    \min_d \sum_{i=1}^N \gamma_i|d_i| \ \text{s.t.} \ d \in \left\{ \begin{aligned} \argmin_{\tilde{d}}\quad & C_{\lambda_2}(\tilde{d}) \\
    \text{s.t.}\quad & x_i + \tilde{d}_i \in \Xi \text{ for all } i \in [N].
    \end{aligned} \right.
\end{aligned}
\end{gather}
We prove the claims \ref{itm:strictly_above_lambda}
and \ref{itm:less_or_equal_lambda} separately.

Proof of claim \ref{itm:strictly_above_lambda}: Inserting the assumption 
$\sum_{i=1}^N \gamma_i|d_i^*| > \Delta$ into \eqref{ineqlam} implies 
$C_{\lambda_1}(d^*) < C_{\lambda_2}(d^*)$ for all $\lambda_1 < \lambda_2$,
which proves $\lambda^*\geq \lambda_2$.

Let $\bar d$ be an arbitrary feasible point for $(\dagger_{\lambda_i})$.
If $\sum_{i=1}^N\gamma_i|\bar d_i| \le \Delta$, then $C_{\lambda_{2}}(d^*)<C_{\lambda_{2}}(\bar d)$
because $C_{\lambda_{2}}(d^*) \ge C_{\lambda_{2}}(\bar d)$ would imply
a violation of the assumed optimality of
$d^*$ for \eqref{eq:minimize_sum_norm_over_solution_set}. 
Because the set of feasible points for the relaxed problem \eqref{eq:relax} has finite size $\vert \Xi \vert ^N$,
we can find an $\varepsilon>0$ such that
\[
C_{\lambda_2}(d^*)+\varepsilon \leq C_{\lambda_2}(d)
\]
holds for all feasible $d$ satisfying $\sum_{i=1}^N\gamma_i|d_i| \le \Delta$.

We define $\lambda_{3} \coloneqq \lambda_2 + \delta$ with $\delta \coloneqq \frac{\varepsilon}{2\Delta}$.
Let $d$ be feasible for the relaxed problem $(\dagger_{\lambda_i})$, then 
\[ C_{\lambda_{3}}(d)>C_{\lambda_{2}}(d)\geq C_{\lambda_{2}}(d^*)
\]
if $\sum_{i=1}^N \gamma_i|d_i| > \Delta$ and
\[ C_{\lambda_{3}}(d) = C_{\lambda_{2}}(d) + \delta\left(\sum_{i=1}^N \gamma_i|d_i|
- \Delta\right)
\geq C_{\lambda_{2}}(d)-\delta \Delta = C_{\lambda_{2}}(d)-\frac{\varepsilon}{2\Delta} \Delta > C_{\lambda_{2}}(d) -\varepsilon \geq C_{\lambda_{2}}(d^*)
\]
if  $\sum_{i=1}^N \gamma_i|d_i| \le \Delta$.
It follows that $\lambda_2 \neq \lambda^*$,
which proves $\lambda^*>\lambda_2$.

Proof of claim \ref{itm:less_or_equal_lambda}: Let
$\sum_{i=1}^N \gamma_i|d_i^*| < \Delta$ hold. We obtain $C_{\lambda_3}(d^*)<C_{\lambda_2}(d^*)$ for arbitrary $\lambda_3 > \lambda_2$
from \eqref{ineqlam}, which proves $\lambda^*\leq \lambda_2$.
\end{proof}

\begin{proposition}
Let \cref{assump:1} hold.
Let $G_c(V_c,A_c)$ be constructed as described in \S\ref{sec:rcspp_formulation}.
Then \cref{alg:bin} terminates after finitely many iterations. Moreover,
the returned value of $\lambda$ differs at most by $\varepsilon$ from an optimal Lagrange multiplier. 
\label{prop:binsearch}
\end{proposition}
\begin{proof}{Proof.}
The only non-trivial operation in each iteration is a shortest path calculation, which terminates finitely. 
The difference between the upper bound $u$ and the lower bound $\ell$ is halved in each iteration.
It follows that $u-\ell \leq \varepsilon$ holds after a finitely many iterations.
Thus \cref{alg:bin} terminates finitely. 

We recall that there is a one-to-one relation between shortest s-t paths in 
$G_c(V_c,A_c)$ with respect to the weight function $\tilde{w}$ defined in
\cref{alg:bin} \cref{ln:lagrangian_relaxation_added} and solutions
of \eqref{eq:opt_lagrange_mult}.

We prove that any optimal Lagrange multiplier is always contained in $[\ell,u]$ inductively over the iterations
of \cref{alg:bin}. \Cref{lemma:bin1} gives that $d\equiv 0$ is an optimal solution of \eqref{eq:relax} for 
$\lambda=c_{\max}+2\alpha$. Combining this with \cref{lemma:bin2} \ref{itm:less_or_equal_lambda} gives that
optimal Lagrange multipliers lie between $0$ and $c_{\max}+2\alpha$, which proves the base claim for the induction.

For an arbitrary iteration we assume that any optimal Lagrange multiplier is in $[\ell,u]$ at the beginning of the
iteration and prove that this still holds after the updates of the bounds provided that \cref{alg:bin}
in this iteration. After each \emph{all shortest paths} calculation, the resource consumption
of a shortest path $p_\lambda$ with minimal resource consumption is checked in \cref{ln:resource_consumption_check}.
We distinguish three cases with respect to the resource consumption of $p_\lambda$.

If the resource consumption of $p_\lambda$ exceeds $\Delta$,
this corresponds to $\sum_{i=1}^N \gamma_i |d_i^*| > \Delta$
for the corresponding solution $d^*$ of \eqref{eq:relax}. 
Note that by choice of $p_\lambda$, the vector $d^*$
minimizes \eqref{eq:minimize_sum_norm_over_solution_set}.
Thus \cref{lemma:bin2} \ref{itm:strictly_above_lambda} implies that
any optimal Lagrange multiplier is above $\lambda$.
Because the lower bound $\ell$ in \cref{alg:bin} is set to $\lambda$
in this case, we obtain that any optimal Lagrange multiplier is still
in $[\ell,u]$.

If the resource consumption equals $\Delta$, then \cref{alg:bin}
terminates early and the optimality of $\lambda$ follows
from the correspondence of optimal solutions and optimal paths
and \cref{prop:welldefined}.

If the resource consumption is strictly less than $\Delta$,
this corresponds to $\sum_{i=1}^N \gamma_i |d_i^*| < \Delta$
for the corresponding solution $d^*$ of \eqref{eq:relax}.
Note again that by choice of $p_\lambda$, the vector $d^*$
minimizes \eqref{eq:minimize_sum_norm_over_solution_set}.
Thus \cref{lemma:bin2} \ref{itm:less_or_equal_lambda} implies that
any optimal Lagrange multiplier is less than or equal
to $\lambda$. Because the upper bound $u$ in \cref{alg:bin} is set to $\lambda$ in this case, we obtain that any
optimal Lagrange multiplier is still in $[\ell,u]$.

This completes the induction and all optimal Lagrange multipliers lie in 
$[\ell,u]$ in all iterations. If the algorithm does not terminate early
with an optimal Lagrange multiplier, the eventual termination with
$u-\ell \le \varepsilon$ implies that the current multiplier, which is
identical to $u$ or $\ell$ by construction differs from an optimal
Lagrange multiplier by at most $\varepsilon$.
\end{proof}

\subsection{Dominated Paths}\label{sec:dominated}

The reformulation as an RCSPP allows us to use the concept of dominated paths
to reduce the search space further.
\begin{definition}
A $[u]$-$[v]$ path $p$ for two nodes $[u]$ and $[v]$ in $G_c$ is called
dominated if there exists another $[u]$-$[v]$ path $q$ with a lower capacity 
consumption such that the cost of the path $p$ is greater than or equal 
to the cost of the path $q$. 
\end{definition}
The equivalence relation allows us to extend this concept to the SPP on $G$. Let $u$ and $v$ be two nodes of $G$ with $u \sim_G v$ and $\tilde r(u)>\tilde r(v)$, then the node $u$ \emph{dominates} the node $v$ if the cost of the shortest path from $s$ to $u$ is less than or equal to the cost of the shortest path from $s$ to $v$. In the graph $G_c$ this corresponds to two paths from $s$ to the same node $[v]$.

The $A^*$ algorithm with a consistent heuristic function guarantees that by the time a node is expanded the shortest path to the node has already been determined. 

\begin{proposition}
Let $h$ be a consistent heuristic function of the $A^*$ algorithm for the graph $G$. Let $u$ and $v$ satisfy $u \sim_G v$ and $\tilde r(u)>\tilde r(v)$. If the node $u$ dominates the node $v$ and the inequality $h(u) \leq h(v)$  holds, then $u$ is expanded before $v$.
\label{prop:domi}
\end{proposition}
\begin{proof}{Proof}
Let $g:V \to \mathbb{R}$ map a node $u$ to the cost of the shortest path from $s$ to $u$ and let $f:V \to \mathbb{R}$ be defined as $f(u) \coloneqq g(u)+h(u)$. In the $A^*$ algorithm the nodes are expanded in increasing order with respect to their values of $f$.
Let $u \sim_G v$ and let $u$ dominate $v$. If both nodes are stored in the priority queue, it follows that $u$ is expanded before $v$ because $f(u)=g(u)+h(u) \leq g(v)+h(v) = f(v)$ holds. The inequality $g(u) \leq g(v)$ follows from the dominance and $h(u) \leq h(v)$ holds due to the assumption. If equality holds, the $A^*$ algorithm expands the node with the higher remaining capacity. Thus $u$ is expanded before $v$ if both nodes are stored in the priority queue.

Therefore, the claim follows if we can exclude the case that $v$ was expanded before $u$ was added to the priority queue. In this case, $v$ was expanded before the predecessor $u_{pred}$ of $u$.
For $u_{pred}$ it follows from the shortest path construction that $g(u_{pred}) \leq g(u)-w(u_{pred},u)$.  
The consistency of the heuristic implies $h(u_{pred}) \leq h(u) + w(u_{pred},u)$. In total, $f(u_{pred}) \leq f(u) \leq f(v)$ follows and $v$ can only be expanded before $u_{pred}$ if $v$ is expanded before $u_{pred}$ is added to the priority queue. By continuing this argumentation inductively $v$ would have to be expanded before $s$ which is not possible. Thus $u$ is expanded before $v$.
\end{proof}
The constructed consistent heuristic function from \cref{prp:final_heuristic}
satisfies the prerequisites of \cref{prop:domi}.
The result ensures that if we expand a node, all dominating nodes have
already been expanded. 
By checking all nodes in the same equivalence class with a higher remaining capacity for dominance, non-promising paths can be discarded early.

This argument can also be applied to each edge in a path because $d_i=0$ is always a feasible choice. Let $u=(\tilde d(u),\tilde r(u),i)$ be a node of the layer $i$ and $v=(\tilde d(v),\tilde r(v),i+1)$ a node of the subsequent layer $i+1$. Then the edge $e=(u,v)$ is not optimal if
\begin{align}
c_{i+1} \tilde d(v)+\alpha \vert x_{i+1}+ \tilde d(v)-x_i- \tilde d(u) \vert - \alpha \vert x_{i+1}-x_i- \tilde d(u) \vert > \alpha \vert \tilde d(v) \vert.
\label{eq:domedge}
\end{align}
The cost of the edge $e$ is $c_{i+1} \tilde d(v)+\alpha \vert x_{i+1}+ \tilde d(v)-x_i- \tilde d(u) \vert$ is compared to the cost of an edge $\tilde e$ from $u$ to a node $(0,\tilde r(u),i+1)$ with cost $\alpha \vert x_{i+1}-x_i- \tilde d(u) \vert$. Because the choice of $d_{i+1}$ can impact the cost of the edges in the next step by no more than  $\alpha \vert \tilde d(v) \vert$, the edge is not optimal if the difference exceeds this value. 

\section{Computational Experiments}\label{sec:computational_experiments}
In order to assess the run times of the proposed graph-based computations, we provide two parameterized
instances of \eqref{eq:p}. We run \cref{alg:slip} on discretizations of them, thereby generating
instances of the problem class \eqref{eq:ip}, which are then solved with 
\begin{enumerate}[leftmargin=2cm]
\item[{\crtcrossreflabel{(Astar)}[itm:astar]}] the $A^*$ algorithm including the Lagrangian-based accelerations described in \S\ref{sec:lagrangian} excluding everything from subsection \S\ref{sec:dominated} but equation $\ref{eq:domedge}$ as the computational cost outweighed the benefits in our experiments,
\item[{\crtcrossreflabel{(TOP)}[itm:top]}] the topological sorting algorithm described in \cite{cormen2009algorithm}, and
\item[{\crtcrossreflabel{(SCIP)}[itm:scip]}] the general purpose IP solver SCIP.
\end{enumerate}
Then we compare the distribution of run times of the three different algorithmic 
solution approaches for solving the generated
instances of \eqref{eq:ip}. Note that it is of course also possible to compare
the run times to Dijkstra's algorithm / the $A^*$ algorithm without the derived
heuristic. However, we have observed very long run times in this case and are
already comparing to two further algorithmic approaches. We have therefore omitted
this additional test case.

The first parameterized IOCP, presented in \S\ref{sec:cex_2nd_elliptic}, is an integer optimal control
problem that is governed by a steady heat equation in one spatial dimension.
We run \cref{alg:slip} on $25$ discretized instances that differ in the choice of the
value for the penalty parameter $\alpha$ (five different choices) and the discretization
constant $N$ (five different choices).

The second parameterized IOCP, presented in \S\ref{sec:cex_signal}, is a generic class of signal reconstruction
problems, which leans on the problem formulation in \cite{kirches2021compactness}. After discretization
it becomes a linear least squares problem with a discrete input variable vector. We assess the performance by
sampling ten of the underlying signal transformation kernels and run \cref{alg:slip} for all of them
for the same choices of the values for $N$ and $\alpha$ as the mixed-integer
PDE-constrained optimization problem, thereby yielding $250$ instances of this problem class.

We present the setup of \cref{alg:slip} in \S\ref{sec:comp_slip_setup}.

\subsection{Integer Control of a Steady Heat Equation}\label{sec:cex_2nd_elliptic}
We consider the class of IOCPs
\begin{gather}\label{eq:iocp_steady}
\begin{aligned}
    \min_{u,x} \quad & \frac{1}{2} \vert \vert u - \bar v \vert \vert_{L^2(-1,1)}^2 + \alpha \TV(x) \\
    \text{s.t.} \quad &  -\varepsilon(t) \frac{d^2 u}{dt^2}(t)=f(t)+x(t) \text{ for a.a.\ } t \in (-1,1),\\
    \quad & u(t)=0 \text{ for } t \in \{-1,1\}, \\
    \quad & x(t) \in \Xi = \{-2, \dots, 23\} \subset \mathbb{Z} \text{ for a.a.\ } t \in (-1,1),
\end{aligned}\tag{SH}
\end{gather}
which is based on \cite[page 112\,ff.]{kouri2012} with the choices $\varepsilon(t)=0.1 \chi_{(-1,0.05)}(t) +10 \chi_{[0.05,)}(t)$,
$f(t)=e^{-(t+0.4)^2}$, and $\bar v(t) = 1$ for all $t \in [-1,1]$. We rewrite the problem in the equivalent reduced form 
\begin{align*}
    \min_{x} \quad & \frac{1}{2} \| S x +u_f-\bar v \vert \vert_{L^2(-1,1)}^2 + \alpha  \TV(x)\\
    \text{s.t.} \quad & x(t) \in \Xi = \{-2, \dots, 23\} \subset \mathbb{Z} \text{ for a.a.\ } t \in (-1,1),
\end{align*}
where the function $S : L^2(-1,1) \to L^2(-1,1)$ denotes the linear solution map of boundary value problem
that constrains \eqref{eq:iocp_steady} for the choice $f = 0$ and $u_f \in L^2(-1,1)$ denotes the solution
of the boundary value problem that constrains \eqref{eq:iocp_steady} for the choice $x = 0$.
With this reformulation and the choice $F(x) \coloneqq \frac{1}{2} \vert \vert u - \bar v \vert \vert_{L^2(-1,1)}^2$
for $x \in L^2(-1,1)$, we obtain the problem formulation \eqref{eq:p}, which gives rise to \cref{alg:slip}
and the corresponding subproblems.

We run \cref{alg:slip} for all combinations
of the parameter values $\alpha \in \{10^{-3}, \dots, 10^{-7}\}$ and uniform discretizations of the domain $(-1,1)$ into
$N \in \{ 512, \ldots, 8192 \}$ intervals. The number of intervals coincides
with the problem size constant $N$ in
the resulting IPs of the form \eqref{eq:ip}.
Because we have a uniform discretization, we obtain $\gamma_i = 1$
for all $i \in [N]$

On each of these 25 discretized instances of \eqref{eq:iocp_steady},
we run our implementation of \cref{alg:slip} with three different initial
iterates $x^0$, thereby giving a total of 75 runs of \cref{alg:slip}. Regarding the initial iterates $x^0$,
we make the following choices:
\begin{enumerate}
\item we compute a solution of the continuous relaxation,
where $\Xi$ is replaced by $\conv \Xi = [-2,23]$ and $\alpha$ is set to zero,
with Scipy's (see \cite{2020SciPy-NMeth}) implementation of limited memory
BFGS with bound constraints (see \cite{liu1989limited})
and round it to the nearest element in $\Xi$ on every interval,
\item we set $x^0(t) = 0$ for all $t \in (-1,1)$, and
\item we compute the arithmetic mean of the two previous choices and round it
to the nearest element in $\Xi$ on every interval.
\end{enumerate}
For each of these instances / runs, we compute the discretization of the least squares term, the PDE, and the
corresponding derivative of $F$ with the help of the open source package Firedrake (see \cite{Rathgeber2016}).
Note that the derivative is computed in Firedrake with the help of so-called adjoint calculus
in a \emph{first-discretize, then-optimize} manner. (see \cite{hinze2008optimization})

\subsection{Generic Signal Reconstruction Problem}\label{sec:cex_signal}
We consider the class of IOCPs
\begin{gather}\label{eq:iocp_signal}
\begin{aligned}
    \min_x \quad & \frac{1}{2} \vert \vert Kx - f \vert \vert_{L^2(0,1)}^2 +  \alpha \TV(x) \\
    \text{s.t.} \quad & x(t) \in \Xi = \{-5, \dots, 5 \} \ \text{ for a.a.\ } t \in (0,1),
\end{aligned}\tag{SR}
\end{gather}
which is already in the form of \eqref{eq:p} with the choice $F(x) \coloneqq \frac{1}{2} \vert \vert Kx - f \vert \vert_{L^2(0,1)}^2$ for $x \in L^2(0,1)$.
In the problem formulation \eqref{eq:iocp_signal} the term
$Kx \coloneqq (k * x)(t) = \int_0^1 k(t-\tau) x(\tau) \dd{\tau} = \int_0^t k(t-\tau) x(\tau) \dd{\tau}$
for $t \in [0,1]$ denotes the convolution of the input control $x$ and a convolution kernel $k$.
We choose the convolution kernel $k$ as a linear combination of $200$ Gaussian kernels
$k(t) = \chi_{[0,\infty)}(t) \sum_{i=1}^{200} b_i \frac{1}{\sqrt{2 \pi} \sigma_i}
\exp\left(\frac{-(t-\mu_i)^2}{2 \sigma_i^2}\right)$ in our numerical experiments.
The coefficients $b_i$, $\mu_i$, and $\sigma_i$ in the Gaussian kernels are sampled from
random distributions, specifically the $c_i$ are drawn from a uniform distribution of values in
[0,1), the $\mu_i$ are drawn from a uniform distribution of values in $[-2,3)$ and the $\sigma_i$
are drawn from an exponential distribution with rate parameter value one.
Furthermore, $f$ is defined as $f(t) \coloneqq 5 \sin(4 \pi t)+10$ for $t \in [0,1]$.

We sample ten kernels $k$ and run \cref{alg:slip} for all combinations of the parameter values
$\alpha \in \{10^{-3}, \dots, 10^{-7}\}$ and
discretizations of the domain $(-1,1)$ into $N \in \{ 512, \ldots, 8192 \}$ intervals.
The number of intervals coincides with the problem size constant $N$ in the resulting IPs of the form \eqref{eq:ip}.
We choose the initial control iterate $x^0(t) = 0$ for all $t \in [0,1]$ for all runs, thereby giving a
total of 250 runs of \cref{alg:slip}.

For each of these instances / runs, we compute a uniform discretization of the least squares term, the operator $K$, and the corresponding derivative of $F$ with the help of a discretization of $[0,1]$ into
$8192$ intervals and approximate the integrals over them with Legendre--Gauss quadrature rules
of fifth order. For choices $N < 8192$ we apply a broadcasting operation of the controls to neighboring intervals
to obtain the corresponding evaluations of $F$ and coefficients in \eqref{eq:ip} with respect to
a smaller number of intervals for the control function ansatz.
We obtain again $\gamma_i = 1$ for all $i \in [N]$ due
to the uniform discretization of the domain $[0,1]$.

\subsection{Computational Setup of \cref{alg:slip}}\label{sec:comp_slip_setup}
For all runs of \cref{alg:slip} we use the reset trust-region radius $\Delta^0 = \frac{1}{8}N$
and the step acceptance ratio $\rho = 0.1$. Regarding the solution of the generated subproblems,
we have implemented both \ref{itm:astar} and \ref{itm:top}
in C++.  For the solution approach \ref{itm:scip} we employ the SCIP Optimization Suite 7.0.3
with the underlying LP solver SoPlex \cite{GamrathEtal2020ZR} to solve the 
IP formulation \eqref{eq:ip}.

For all IOCP instances with $N \in \{512, \ldots, 2048 \}$ we prescribe a 
time limit of 240 seconds for the IP solver but note that almost all instances require only between a fraction of a second
and a single digit number of seconds computing time for global optimality so that the time limit is
only reached for a few cases.

For higher numbers of $N$ the running times of the subproblems are generally 
too long to be able to solve all instances with a meaningful time limit for
the solution approach \ref{itm:scip}. In order to assess the run times for 
\ref{itm:scip} in this case as well we draw $50$ instances of \eqref{eq:ip}
from the pool of all generated subproblems both for $N = 4096$ and $N = 8192$
and solve them with a time limit of one hour each. Moreover, we do the same 
with the $50$ instances of \eqref{eq:ip} for which the solution approach \ref{itm:astar} has the longest run times.

We note that a run of \cref{alg:slip} may return different
results depending on which solution approach is used for
the trust-region subproblems even if all trust-region 
subproblems are solved optimally because the minimizers
need not to be unique and different algorithms may find
different minimizers. In order to be able to compare
the performance of the algorithms properly we record the
trust-region subproblems that are generated when using
\ref{itm:astar} as subproblem solver and pass the resulting
collection of instances of \eqref{eq:ip} to the
solution approaches \ref{itm:top} and \ref{itm:scip}.

A laptop computer with an Intel(R) Core i7(TM) CPU with eight cores
clocked at 2.5\,GHz and 64 GB RAM serves as the computing platform
for all of our experiments.

\section{Results}\label{sec:numresults}
We present the computational results of both IOCPs. We provide run times
for each of the three algorithmic solution approaches \ref{itm:astar}, 
\ref{itm:top}, and \ref{itm:scip}.
The 75 runs of \cref{alg:slip} on instances of 
\eqref{eq:iocp_steady}, see \S\ref{sec:cex_2nd_elliptic},
generate 30440 instances of \eqref{eq:ip} in total.
The 250 runs of \cref{alg:slip} on instances of 
\eqref{eq:iocp_steady}, see \S\ref{sec:cex_signal},
generate 655361 instances of \eqref{eq:ip} in total.
A detailed tabulation including a break down with
respect to the values of $\alpha$ and $N$
can be found in \cref{tbl:number_of_generated_ips}
in \S\ref{sec:tabulated_data}.

We analyze the recorded run times of all solution approaches with respect
to the value $N$ in \S\ref{sec:run_times_wrt_N}.
The run times of \ref{itm:astar} and \ref{itm:top} turn out to
be generally much lower than those of \ref{itm:scip}
and we analyze their run times in more detail with respect to
the product of $N$ and $\Delta$ and the number of nodes
in the graph in \S\ref{sec:run_times_wrt_deltatimesN}.
Finally, we assess the cumulative time required for the subproblem
solves of the runs of \cref{alg:slip} when choosing between
\ref{itm:top} and \ref{itm:astar} as subproblem solver
depending on the value of $\Delta$ in \S\ref{sec:subproblem_solver_choice}.

\subsection{Run Times with respect to $N$}\label{sec:run_times_wrt_N}
All instances of \eqref{eq:ip} are solved in a couple of seconds with the
solution approaches \ref{itm:astar} and \ref{itm:top}. The solution
approach \ref{itm:scip} solves almost all instances within the prescribed
time limits specified in \S\ref{sec:comp_slip_setup}.
The exceptions are one instance with $N = 1024$ (out of 5268),
14 instances with $N = 2048$ (out of 5680), and two instances with
$N = 4096$ (out of 100)
for \eqref{eq:iocp_steady} and
one instance with $N = 8192$ (out of 100) for 
\eqref{eq:iocp_signal}.

For \ref{itm:astar}, the mean run time to solve the \eqref{eq:ip} 
instances generated for \eqref{eq:iocp_steady} increases from 
\num{0.016}\,s to \num{2.1}\,s over the increase of $N$ from
\num{512} to \num{8192}. For \ref{itm:top}, the mean run time 
starts from the lower value \num{0.015}\,s at
$N = 512$, surpasses the mean run time of \ref{itm:astar}
for $N = 1024$ and increases to the higher value \num{4.4}\,s
for $N = 8192$, all other things being equal.
For \ref{itm:scip}, the mean run time of 
increases from \num{0.51}\,s for $N = 512$ to
\num{4.7}\,s for $N = 2048$. For $N = 4096$, the mean run
time of \ref{itm:scip} is \num{210}\,s for
the fifty randomly drawn instances and \num{128}\,s for
the fifty instances, for which \ref{itm:astar} has
the highest run times.
For $N = 8192$, the mean run time of \ref{itm:scip}
is \num{56}\,s for the fifty randomly drawn instances
and \num{598}\,s for the fifty instances, for which
\ref{itm:astar} has the highest run times.

We note that \ref{itm:scip} did not solve two of
the randomly drawn instances for $N = 4096$ within
a one hour time limit, thereby affecting the mean 
significantly (it would be around \num{68}\,s without these
two instances).

We obtain a similar picture for the  \eqref{eq:ip} instances
generated for \eqref{eq:iocp_signal} although with generally
lower run times. For \ref{itm:astar}, the mean run time to solve the
instances generated for \eqref{eq:iocp_steady} increases from 
\num{0.011}\,s to \num{0.23}\,s over the increase of $N$ from
\num{512} to \num{8192}. For \ref{itm:top}, the mean run time 
generated starts from the lower value \num{0.0054}\,s at
$N = 512$, surpasses the mean run time of \ref{itm:astar}
for $N = 2048$ and increases to the higher value \num{0.68}\,s
for $N = 8192$, all other things being equal.
In the same setting, the mean run time of \ref{itm:scip}
increases from \num{0.075}\,s for $N = 512$
to \num{0.72}\,s for $N = 2048$.
For $N = 4096$, the mean run time of \ref{itm:scip}
is \num{5.8}\,s for
the fifty randomly drawn instances and \num{9.8}\,s for
the fifty instances, for which \ref{itm:astar} has
the highest run times.
For $N = 8192$, the mean run time of \ref{itm:scip}
is \num{11}\,s for the 50 randomly drawn instances
and \num{309}\,s for the 50 instances, for which
\ref{itm:astar} has the highest run times.

We illustrate the distribution of the run times for the different solution
approaches and the different values of $N$ with violin plots
in \cref{fig:violinplots}.
A detailed tabulation of the mean run times with a break down for
the different values of $N$ and the different algorithmic solution
approaches for \eqref{eq:ip} can be found in
\cref{tbl:mean_run_time} in \S\ref{sec:tabulated_data}.

\begin{figure}
\hspace{-0.5cm}
\begin{subfigure}{1.0\textwidth}
\centering
\includegraphics[scale = 0.295]{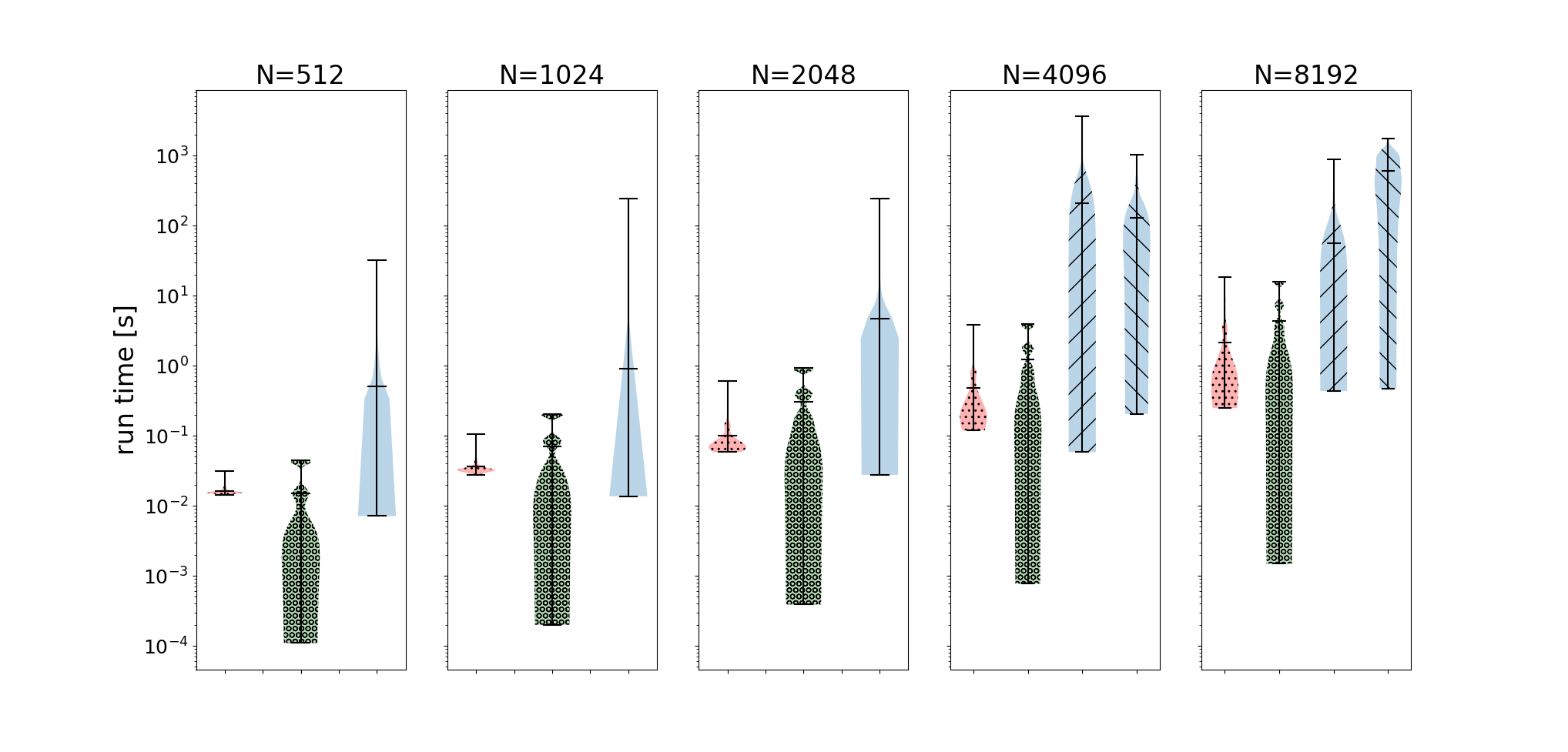}
\end{subfigure}
\begin{subfigure}{1.0\textwidth}
\centering
\includegraphics[scale = 0.295]{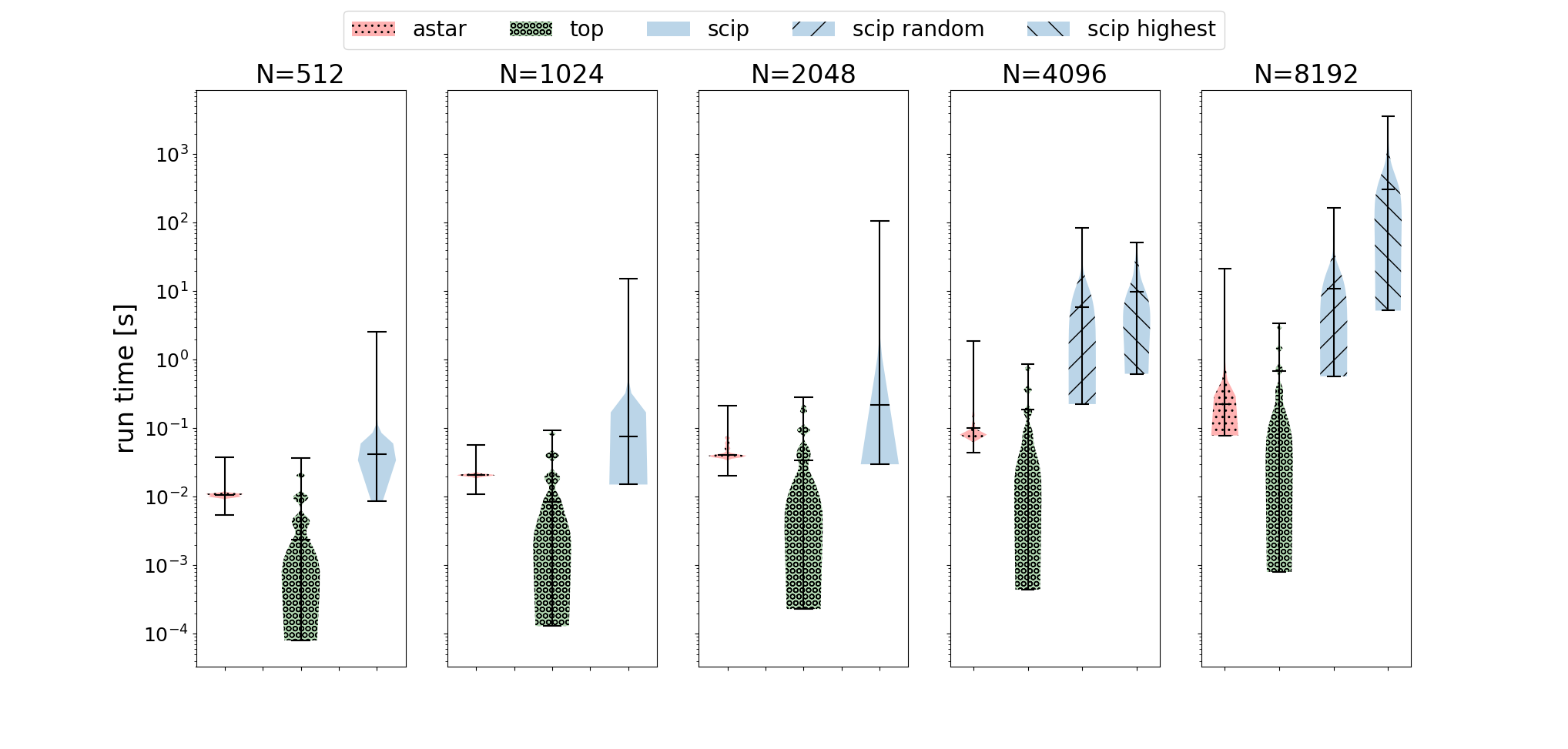}
\end{subfigure}
\caption{Distribution of the run times for the solution of the generated instances
of \eqref{eq:ip} with \ref{itm:astar} (left), \ref{itm:top} (center), and \ref{itm:scip} (right)
as violin plots with range limits and mean marked with black strokes
for the subproblems originating from \eqref{eq:iocp_steady} (top)
and \eqref{eq:iocp_signal} (bottom). For $N \in \{4096,8192\}$ the
results of the solution approach \ref{itm:scip} are split into the
results for the randomly chosen instances ("scip random", center right)
and the instances on which \ref{itm:astar} has the highest run times
("scip worst", right), see \S\ref{sec:comp_slip_setup}.}
\label{fig:violinplots}
\end{figure}

\subsection{Run Times of \ref{itm:top} and \ref{itm:astar}
with respect to $N \Delta$ and the Graph 
Size}\label{sec:run_times_wrt_deltatimesN}
Because $\Xi$ in \eqref{eq:ip} depends only on the superordinate IOCP
and does not vary over a run of \cref{alg:slip}, we consider $\Xi$ as
fixed. Then the complexity of \ref{itm:top} depends linearly on
the number of nodes and edges in the graph and thus the product
of $N$ and the trust-region radius $\Delta$, see \S\ref{sec:solution algo}. 
Moreover, \ref{itm:top} does not have any option to
skip nodes or terminate early. In contrast to this, \ref{itm:astar} can
make use of the heuristic function
and the preprocessing described in \S\ref{sec:heuristic}-\ref{sec:preprocessing}.
Therefore, we assess the run times of \ref{itm:top} and \ref{itm:astar}
with respect to the problem size of \eqref{eq:ip} measured as $N \Delta$.

We observe that the mean run times produced by \ref{itm:top} follow
approximately a linear trend with respect to $N\Delta$ starting from 
very low values at the order of $10^{-4}$\,s for $N \Delta \approx 10^4$
increasing to values at the order of $10^1$\,s for
$N\Delta \approx 2\cdot 10^8$. In contrast to this, the mean run times
of \ref{itm:astar} start at the order of $10^{-2}$\,s
for $N \Delta \approx 10^4$, follow a generally shallower
but (at first impression less linear)
trend and are about an order of magnitude lower
for the highest values of $N\Delta$.
In particular, the run time of \ref{itm:top} starts exceeding
the run time of \ref{itm:astar} between $N\Delta = 10^6$
and $N\Delta = 10^7$ for the subproblems generated with
both IOCP instances \eqref{eq:iocp_steady} and \eqref{eq:iocp_signal}.
The mean run times over the different values of $N\Delta$ are
plotted in \cref{fig:product}.

\begin{figure}[H]
\hspace*{-1.5cm}  \begin{subfigure}{0.56\textwidth}
\centering
\includegraphics[scale = 0.2]{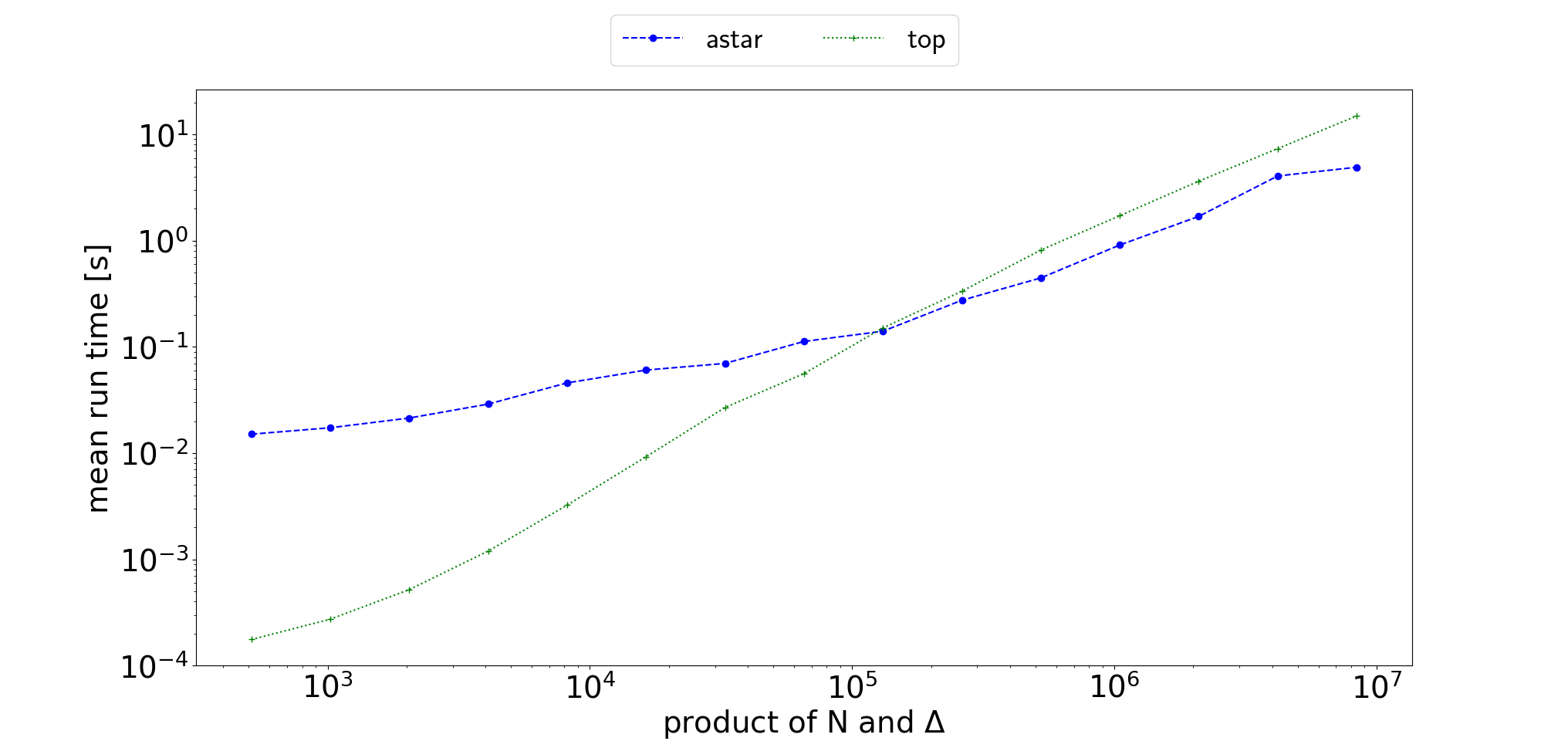}
\end{subfigure}
\begin{subfigure}{0.5\textwidth}
\centering
\includegraphics[scale = 0.2]{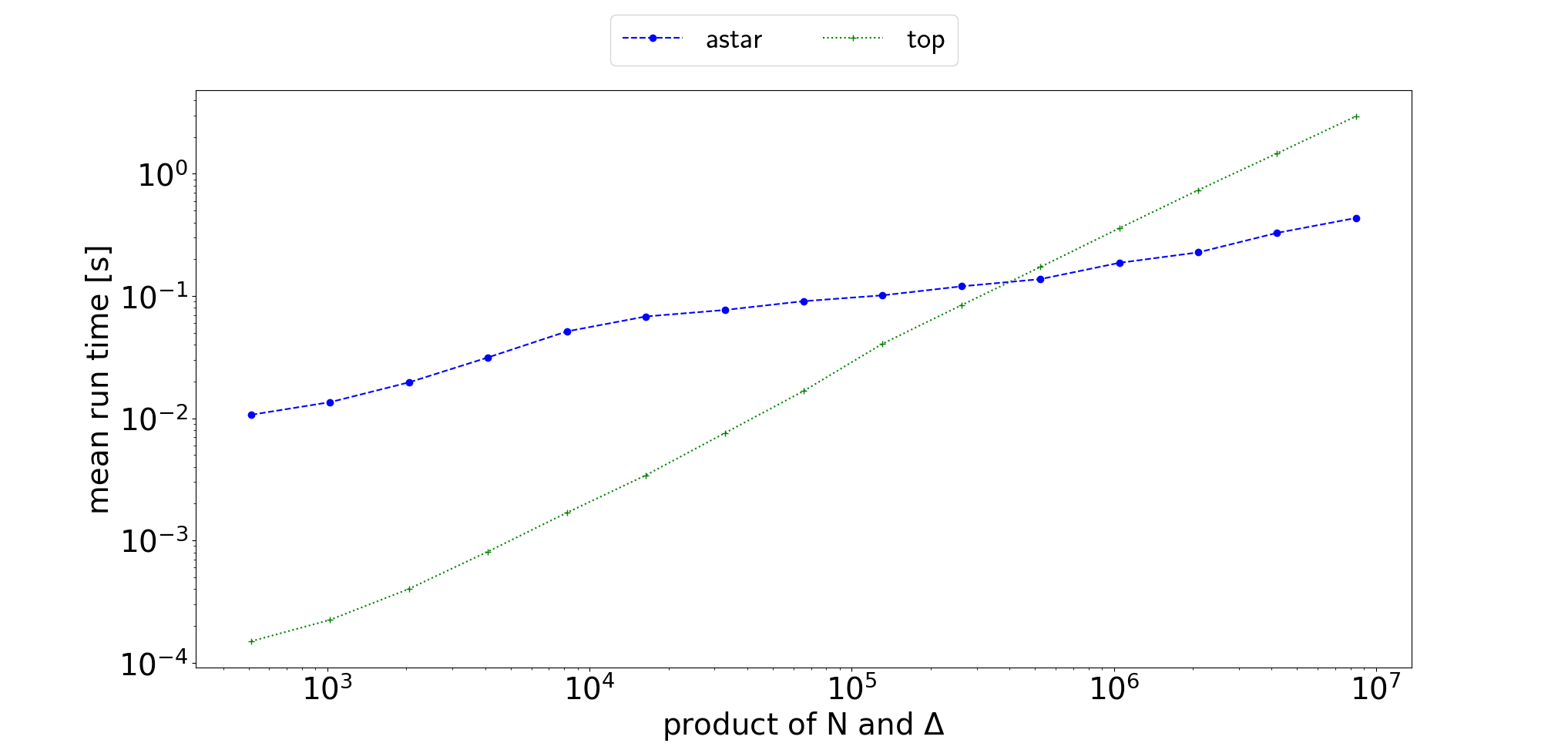}
\end{subfigure}
\caption{Mean run times of \ref{itm:top} (green, + mark, dotted)
and \ref{itm:astar} (blue, circle mark, dashed) over the product of the 
discretization $N$ and the trust-region radius $\Delta$ for
\eqref{eq:ip} instances stemming from
\eqref{eq:iocp_steady}  (left) and \eqref{eq:iocp_signal}
(right).}
\label{fig:product}
\end{figure}

We get a similar picture for the overall trend when considering
the run times of \ref{itm:top} and \ref{itm:astar} with respect
to the number of nodes in the graph. However, the run times do not
increase monotonically over the number of nodes in the graph
and we observe a sequence of spikes in the run times, which is
illustrated in \cref{fig:size}. We attribute the
dominant spikes in the run times of \ref{itm:astar} to the
preprocessing step described in \S\ref{sec:preprocessing}.
The run time of the preprocessing step 
\cref{alg:bin} yields an offset for the run time of
\ref{itm:astar}, which only depends
on $N$ and not on the current trust-region radius $\Delta$.

We also observe a sequence of spikes in the overall run time trend
of \ref{itm:top} with respect to the product $N \Delta$. The locations of
these spikes seem to be opposed to spikes we observe for \ref{itm:astar}.
We attribute these spikes to the fact that if $N$ is relatively
large compared to $\Delta$, then the number of edges in the graph is
relatively small compared to a graph with a similar value of $N\Delta$,
where the ratio $\tfrac{N}{\Delta}$ is smaller.

\begin{figure}[H]
\hspace*{-1.5cm} \begin{subfigure}{0.56\textwidth}
\centering
\includegraphics[scale = 0.2]{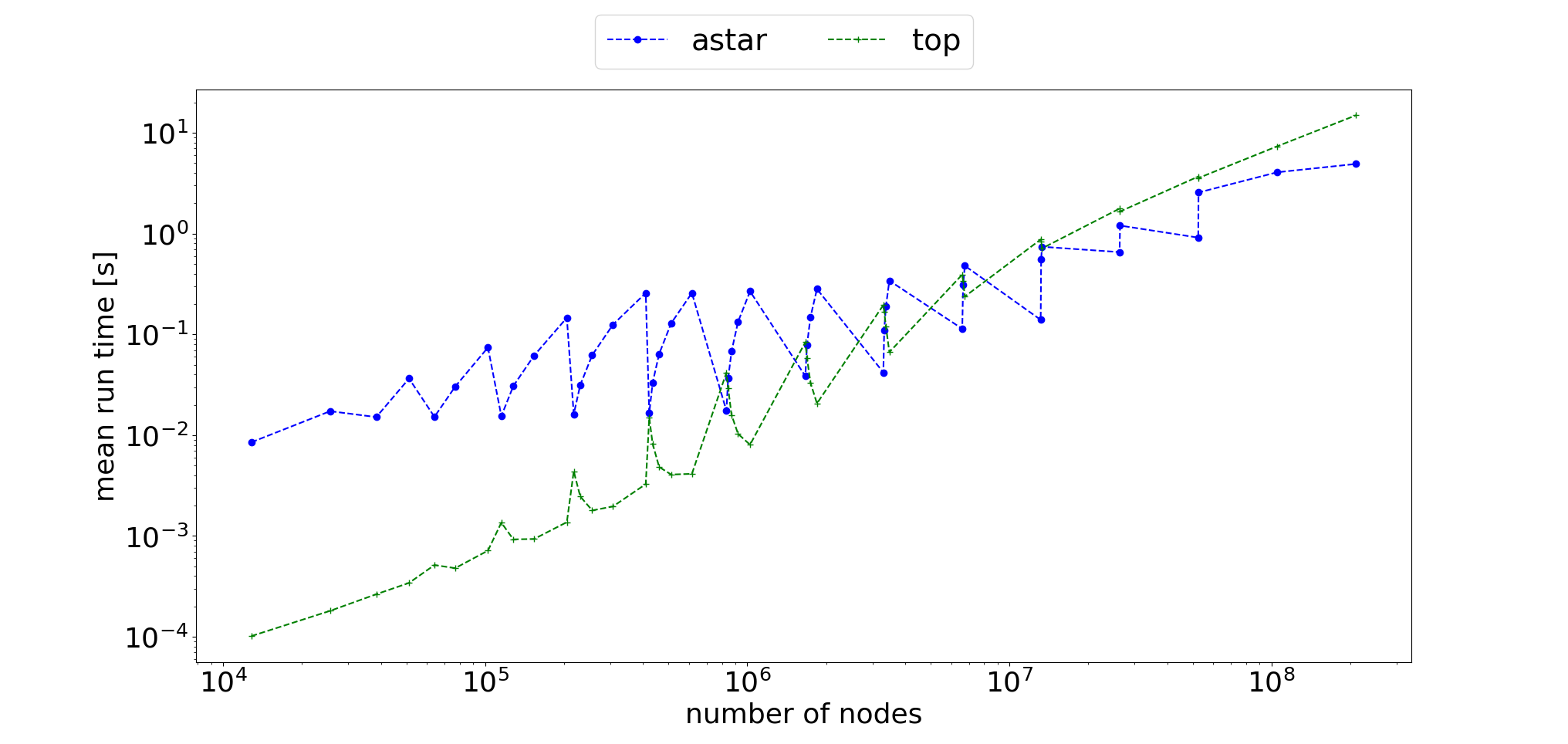}
\end{subfigure}
\begin{subfigure}{0.5\textwidth}
\centering
\includegraphics[scale = 0.2]{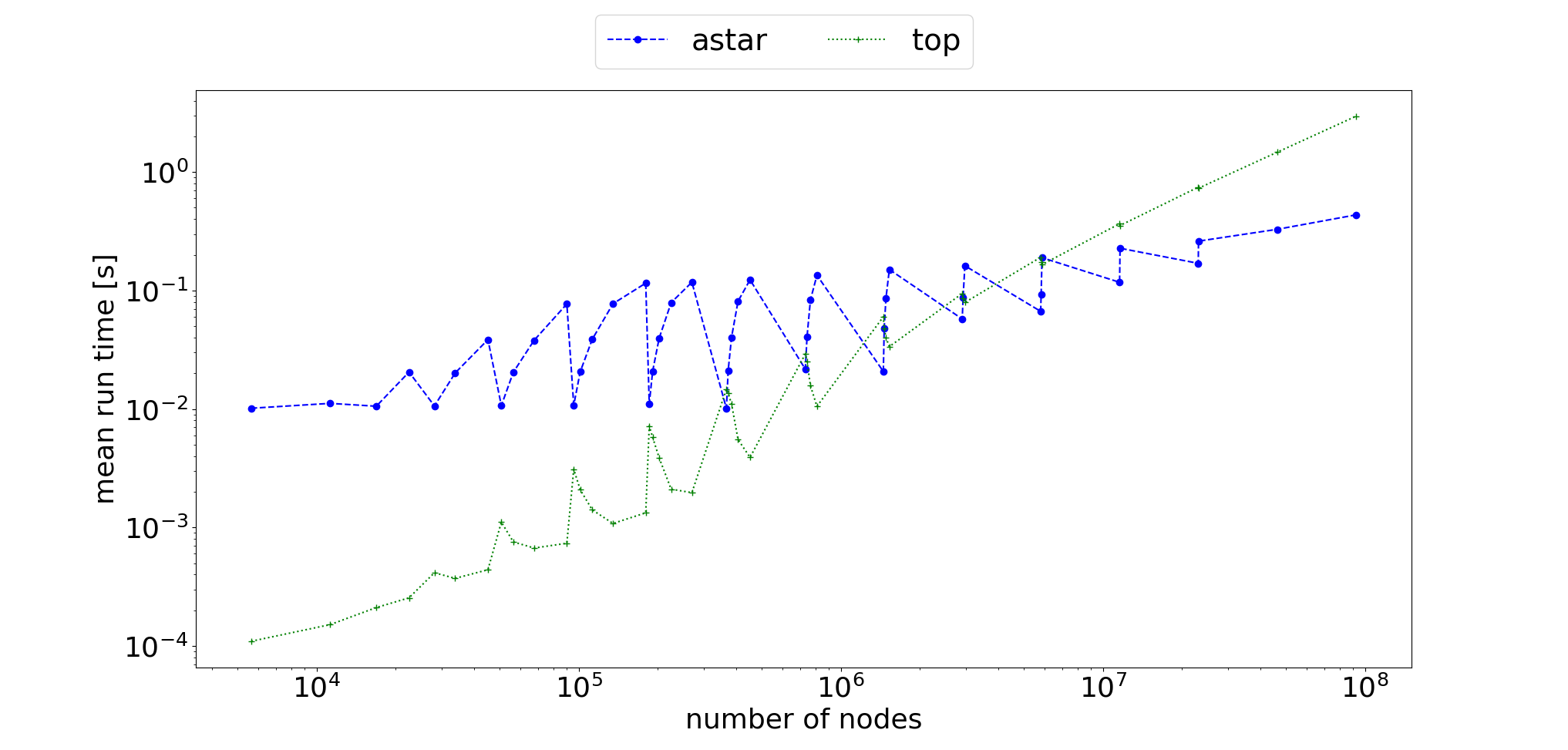}
\end{subfigure}
\caption{Mean run times of \ref{itm:top} (green, + mark, dotted)
and \ref{itm:astar} (blue, circle, dashed) over the number of nodes
in the graphs stemming for \eqref{eq:ip} instances
stemming from \eqref{eq:iocp_steady} (left) and \eqref{eq:iocp_signal} (right).}
\label{fig:size}
\end{figure}

A closer investigation yields that the run times of \ref{itm:astar}
increase less than linearly with respect to $\Delta$. Specifically, 
\ref{itm:astar} expands only a smaller fraction of the nodes in
the graph and larger trust-region radii result in smaller
percentages of the nodes being expanded. In particular
for the instances generated from \eqref{eq:iocp_signal}, 
the average percentage of nodes expanded is less than $0.2 \%$
for the largest trust-region radius $\Delta = \tfrac{1}{8}N$.
We visualize the mean and the 95th percentile of the nodes expanded for
different trust-region radii in \cref{fig:deviation}
in \S\ref{sec:auxiliary_figures}.

Considering the time required for whole runs of \cref{alg:slip},
we observe a substantial decrease when comparing the runs with 
\ref{itm:top} and \ref{itm:astar} as subproblem solver to the same
runs with \ref{itm:scip} as subproblem solver. The strongest
effect can be observed for $N = 2048$, the largest case, where
the necessary data is fully available.
For $N = 2048$ and \eqref{eq:iocp_steady}
we observe a decrease of the cumulative run time
of the runs of \cref{alg:slip} from $30200$\,s with \ref{itm:scip}
to $5242$\,s with \ref{itm:top} and $4083$\,s with \ref{itm:astar}.
Thus the subproblem solves consume $88.4$\,\% of the run time
of \cref{alg:slip} for \ref{itm:scip}, $32.9$\,\% for 
\ref{itm:top} an $13.9$\,\% for \ref{itm:astar}.
For $N = 2048$ and \eqref{eq:iocp_signal}
we observe a decrease of the cumulative run time
of the runs of \cref{alg:slip} from $85696$\,s with \ref{itm:scip}
to $24914$\,s with \ref{itm:top} and $24183$\,s with \ref{itm:astar}.
Thus the subproblem solves consume $76.9$\,\% of the run time
of \cref{alg:slip} for \ref{itm:scip}, $20.5$\,\% for 
\ref{itm:top} an $18.1$\,\% for \ref{itm:astar}. We have tabulated
the cumulative run times of the runs of \cref{alg:slip} excluding
the subproblem solvers over $N$ and $\alpha$
in \cref{tbl:run_time_slip} and the cumulative run times
of the subproblem solves over $N$ and $\alpha$
for \ref{itm:astar} and \ref{itm:top} in
in \cref{tbl:run_time_astar,tbl:run_time_top}.

\subsection{Choosing the Subproblem Solver Depending on $\Delta$}\label{sec:subproblem_solver_choice}
We choose the subproblem solver in \cref{alg:slip} depending on the
value of $\Delta$. Specifically, we choose \ref{itm:top} if $\Delta$
is less than a decision value $\Delta_D$ and \ref{itm:astar} if
$\Delta$ is greater than or equal to $\Delta_D$.
For all values of $N$, we observe that the cumulative run time
of the subproblem solves in \cref{alg:slip} decreases from the
value obtained if only \ref{itm:top} is used if $\Delta_D$
is decreased until $\Delta_D$ is between $64$ and $128$.
Then the cumulative run time increases until it reaches the value
that is obtained if only \ref{itm:astar} is used ($\Delta_D = 0$).
We visualize the cumulative run times
for the subproblem solves of the runs of \cref{alg:slip} for
$N \in \{512,2048,8192\}$
for decreasing values of $\Delta_D$ from $\Delta_D = \tfrac{1}{4}N$
(only \ref{itm:top}) to $\Delta_D = 0$ (only \ref{itm:astar})
\cref{fig:split} in \S\ref{sec:auxiliary_figures}. 

\section{Conclusion}\label{sec:conclusion}
We have derived a reformulation of the problem class \eqref{eq:ip},
which arises as discretized trust-region subproblem in implementations
of \cref{alg:slip}, as an RCSPP on an LDAG and as well as a Lagrangian
relaxation of the (discretized) trust-region constraint.

The reformulation and the Lagrangian relaxation have lead to a highly
efficient $A^*$ algorithm that provides optimal solutions and which
outperforms on average a general purpose IP solver generally and a
topological sorting algorithm on larger problem instances despite
having a worse complexity estimate than topological sorting.
The average run times per subproblem for both topological sorting
and $A^*$ are several orders of magnitude lower than those for
a general purpose IP solver. We also note that we have observed
occasionally that the general purpose IP solver is not able to
solve the IP to global optimality within meaningful time limits.
Using our $A^*$ algorithm or topological sorting allows to overturn
the relationship between the run time of the subproblem solves
and the other operations in \cref{alg:slip},
mainly the computations of $F(x^k)$ and $\nabla F(x^k)$.

We believe that the better performance of $A^*$ can be attributed
to the fact that the preprocessing and heuristic, both derived from the
Lagrangian relaxation, allow to disregard large parts of the underlying
graph while the topological sorting always has to process every node
and edge in the graph, which is backed by our observation that the
percentage of expanded nodes by $A^*$ is particularly low if $N$
is large.
This observation also leads to a potential further improvement of
the run time of \cref{alg:slip} because we may decide on using
topological sorting or $A^*$ as subproblem solver depending
on the current value of $\Delta$.

\section*{Acknowledgment}
The authors are grateful to Christoph Hansknecht
(Technical University of Braunschweig) for helpful
discussions and advice on 
the implementation of the algorithms.
\begin{appendices} 
\section{Tabulated Data of the Numerical Results}\label{sec:tabulated_data}

\begin{table}[H]
\caption{Number of generated IPs of the form \eqref{eq:ip}
by the runs of \cref{alg:slip} on the different discretizations
and parameterizations
of \eqref{eq:iocp_steady} (left)
and \eqref{eq:iocp_signal} (right).}\label{tbl:number_of_generated_ips}
\parbox{.475\linewidth}{
\begin{center}
\small
\begin{tabular}{r|rrrrr|r} 
 \toprule
 \diagbox{$\alpha$}{N} & 512 & 1024 & 2048 & 4096 & 8192 & Cum. \\
 \midrule
 $10^{-7}$ & 2055 & 2960 & 3692 & 4645 & 3878 & 17230 \\ 
 $10^{-6}$ & 1727 & 1291 & 639 & 507 & 1306 & 5470 \\
 $10^{-5}$ & 310 & 124 & 339 & 524 & 1018 & 2315 \\
 $10^{-4}$ & 365 & 296 & 427 & 345 & 986 & 2419 \\
 $10^{-3}$ & 303 & 597 & 583 & 675 & 848 & 3006 \\
 \midrule
 Cum. & 4760 & 5268 & 5680 & 6696 & 8036 & 30440\\
 \bottomrule
\end{tabular}
\end{center}}\hfill
\parbox{.5\linewidth}{
\begin{center}
\small
\begin{tabular}{r|rrrrr|r} 
 \toprule
 \diagbox{$\alpha$}{N} & 512 & 1024 & 2048 & 4096 & 8192 & Cum. \\
 \midrule
 $10^{-7}$ & 5382 & 10026 & 19026 & 35294 & 71696 & 141424\\ 
 $10^{-6}$ & 5142 & 10713 & 18821 & 36514 & 75515 & 146705\\
 $10^{-5}$ & 5119 & 10101 & 19526 & 34843 & 60932 & 130521\\
 $10^{-4}$ & 5199 & 9202 & 17737 & 38816 & 63021 & 133975\\
 $10^{-3}$ & 3899 & 8875 & 16291 & 32812 & 40859 & 102736\\
 \midrule
 Cum. & 24741 & 48917 & 91401 & 178279 & 312023 & 655361\\
 \bottomrule
\end{tabular}
\end{center}}
\end{table}

\begin{table}[H]
\caption{Cumulative run times in seconds of \cref{alg:slip} excluding the run times 
required for the subproblems of the form \eqref{eq:ip} for all instances and start 
values on the different discretizations
and parameterizations
 of \eqref{eq:iocp_steady} (left)
and \eqref{eq:iocp_signal} (right).}\label{tbl:run_time_slip}
\parbox{.475\linewidth}{
\begin{center}
\small
\begin{tabular}{r|rrrrr|r} 
 \toprule
 \diagbox{$\alpha$}{N} & 512 & 1024 & 2048 & 4096 & 8192 & Cum. \\
 \midrule
 $10^{-7}$ & 439 & 1151 & 2147 & 3770 & 4056 & 11563 \\ 
 $10^{-6}$ & 450 & 585 & 456 & 429 & 1299 & 3219 \\
 $10^{-5}$ & 82 & 50 & 219 & 448 & 1066 & 1865 \\
 $10^{-4}$ & 99 & 137 & 288 & 313& 1039 & 1876 \\
 $10^{-3}$ & 91 & 271 & 405 & 598 & 997 & 2362\\
 \midrule
 Cum. & 1161 & 2194 & 3515 & 5558& 8457 & 20885\\
 \bottomrule
\end{tabular}
\end{center}}\hfill
\parbox{.5\linewidth}{
\begin{center}
\small
\begin{tabular}{r|rrrrr|r} 
 \toprule
 \diagbox{$\alpha$}{N} & 512 & 1024 & 2048 & 4096 & 8192 & Cum. \\
 \midrule
 $10^{-7}$ & 1175 & 2206 & 4109 & 7774 & 16455 & 31719\\ 
 $10^{-6}$ & 1135 & 2353 & 4075 & 8075 & 17166 & 32804\\
 $10^{-5}$ & 1127 & 2220 & 4224 & 7744 & 16691 & 32006\\
 $10^{-4}$ & 1143 & 2011 & 3852 & 8569 & 18260 & 33835\\
 $10^{-3}$ & 861 & 1958 & 3536 & 7211 & 11390 & 24956\\
 \midrule
 Cum. & 5441 & 10748& 19796 & 39373 & 79962& 155320\\
 \bottomrule
\end{tabular}
\end{center}}
\end{table}

\begin{table}[H]
\caption{Cumulative run times in seconds of \ref{itm:astar} for the subproblems of the form \eqref{eq:ip} for all instances and start 
values on the different discretizations
and parameterizations
 of \eqref{eq:iocp_steady} (left)
and \eqref{eq:iocp_signal} (right).}\label{tbl:run_time_astar}
\parbox{.475\linewidth}{
\begin{center}
\small
\begin{tabular}{r|rrrrr|r} 
 \toprule
 \diagbox{$\alpha$}{N} & 512 & 1024 & 2048 & 4096 & 8192 & Cum. \\
 \midrule
 $10^{-7}$ & 32 & 97 & 313 & 2073 & 8258 & 10773 \\ 
 $10^{-6}$ & 29 & 52 & 72 & 279 & 2540 & 2972 \\
 $10^{-5}$ & 5 & 4 & 42 & 251 & 2378 & 2680 \\
 $10^{-4}$ & 6 & 12 & 58 & 216& 1857 & 2149 \\
 $10^{-3}$ & 6 & 26 & 84 & 370 & 2019 & 2505\\
 \midrule
 Cum. & 78 & 191 & 569 & 3189& 17052 & 21079\\
 \bottomrule
\end{tabular}
\end{center}}\hfill
\parbox{.5\linewidth}{
\begin{center}
\small
\begin{tabular}{r|rrrrr|r} 
 \toprule
 \diagbox{$\alpha$}{N} & 512 & 1024 & 2048 & 4096 & 8192 & Cum. \\
 \midrule
 $10^{-7}$ & 58 & 210 & 881 & 3373 & 12544 & 17066\\ 
 $10^{-6}$ & 54 & 224 & 885 & 3501 & 13387 & 18051\\
 $10^{-5}$ & 54 & 211 & 926 & 3387 & 14450 & 19028\\
 $10^{-4}$ & 55 & 192 & 875 & 4142 & 24823 & 30087\\
 $10^{-3}$ & 42 & 184 & 801 & 3516 & 14977 & 19520\\
 \midrule
 Cum. & 263 & 1021& 4368 & 17919 & 80181& 103752\\
 \bottomrule
\end{tabular}
\end{center}}
\end{table}

\begin{table}[H]
\caption{Cumulative run times in seconds of \ref{itm:top} for the subproblems of the form \eqref{eq:ip} for all instances and start 
values on the different discretizations
and parameterizations
 of \eqref{eq:iocp_steady} (left)
and \eqref{eq:iocp_signal} (right).}\label{tbl:run_time_top}
\parbox{.475\linewidth}{
\begin{center}
\small
\begin{tabular}{r|rrrrr|r} 
 \toprule
 \diagbox{$\alpha$}{N} & 512 & 1024 & 2048 & 4096 & 8192 & Cum. \\
 \midrule
 $10^{-7}$ & 29 & 201 & 1070 & 5614 & 18046 & 24960\\ 
 $10^{-6}$ & 28 & 100 & 231 & 612 & 4859 & 5830 \\
 $10^{-5}$ & 5 & 7 & 98 & 637 & 4263 & 5010 \\
 $10^{-4}$ & 5 & 23 & 136 & 467& 3949 & 4580\\
 $10^{-3}$ & 5 & 41 & 190 & 836 & 4028 & 5100\\
 \midrule
 Cum. & 72 & 372 & 1725 & 8166 & 35145 & 45480\\
 \bottomrule
\end{tabular}
\end{center}}\hfill
\parbox{.5\linewidth}{
\begin{center}
\small
\begin{tabular}{r|rrrrr|r} 
 \toprule
 \diagbox{$\alpha$}{N} & 512 & 1024 & 2048 & 4096 & 8192 & Cum. \\
 \midrule
 $10^{-7}$ & 28 & 192 & 1061 & 6798 & 49089 & 57168\\ 
 $10^{-6}$ & 27 & 200 & 1067 & 6943 & 50870 & 59107\\
 $10^{-5}$ & 27 & 198 & 1089 & 6654 & 49414 & 57382\\
 $10^{-4}$ & 28 & 177 & 1007 & 7153 & 54746 & 63111\\
 $10^{-3}$ & 23 & 174 & 925 & 6232 & 35243 & 42597\\
 \midrule
 Cum. & 133 & 941& 5149 & 33780 & 239362& 279365\\
 \bottomrule
\end{tabular}
\end{center}}
\end{table}

\begin{table}[H]
\caption{Mean run times in seconds of the solution process for
the IPs of the form \eqref{eq:ip} for different solution
algorithms and problem sizes. The values for the IPs generated
by running \cref{alg:slip} on the different discretizations
and parameterizations of \eqref{eq:iocp_steady} are tabulated
left and those
of \eqref{eq:iocp_signal} are tabulated right. The smallest
mean run time per value of $N$ is written in bold type.
For the algorithm \ref{itm:scip} the mean is computed over
all instances (all) for $N \in \{512,1024,2048\}$
and over 50 randomly drawn instances (random)
as well as the 50 instances with the highest run time
of \ref{itm:astar} (worst) for $N \in \{4096,8192\}$.}\label{tbl:mean_run_time}
\parbox{.475\linewidth}{
\begin{center}
\small
\begin{tabular}{r|lllll} 
 \toprule
 N
 & \ref{itm:astar}
 & \ref{itm:top}
 & \multicolumn{3}{c}{\ref{itm:scip}}\\
 & & &  all & random & worst\\ [0.5ex]
\midrule
512 & 0.016 & \textbf{0.015} &0.511& - & - \\
1024& \textbf{0.037} & 0.071 &0.901 & - & -\\
2048 & \textbf{0.100} & 0.304 & 4.698 & - & -\\
4096 & \textbf{0.476} & 1.220 & - & 209.751 & 128.430 \\
8192 & \textbf{2.122} & 4.374 & - & 56.023 & 597.785 \\
 \bottomrule
\end{tabular}
\end{center}}\hfill
\parbox{.5\linewidth}{
\begin{center}
\small
\begin{tabular}{r|lllll} 
 \toprule
 N
 & \ref{itm:astar}
 & \ref{itm:top}
 & \multicolumn{3}{c}{\ref{itm:scip}}\\
 & & &  all & random & worst\\ [0.5ex]
\midrule
512 & 0.011 & \textbf{0.005} & 0.075 & - & - \\
1024 & 0.021 & \textbf{0.019} & 0.223 & - & -\\
2048 & \textbf{0.048} & 0.056 & 0.721 & - & -\\
4096 & \textbf{0.101} & 0.190 & - & 5.822 & 9.807 \\
8192 & \textbf{0.227} & 0.677 & - & 10.925 & 309.499 \\
\bottomrule
\end{tabular}
\end{center}}
\end{table}

\section{Auxiliary Figures Illustrating the Numerical Results}\label{sec:auxiliary_figures}
\end{appendices}

\begin{figure}[H]
\hspace*{-1cm} \begin{subfigure}{0.55\textwidth}
\centering
\includegraphics[scale = 0.18]{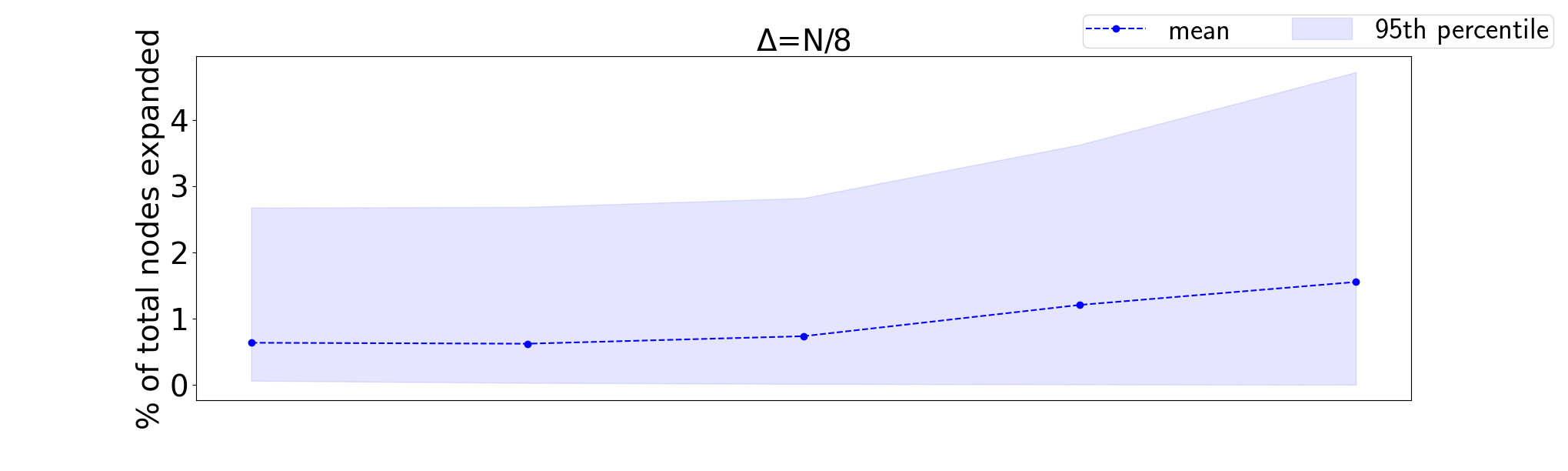}
\end{subfigure}
\begin{subfigure}{0.5\textwidth}
\centering
\includegraphics[scale = 0.18]{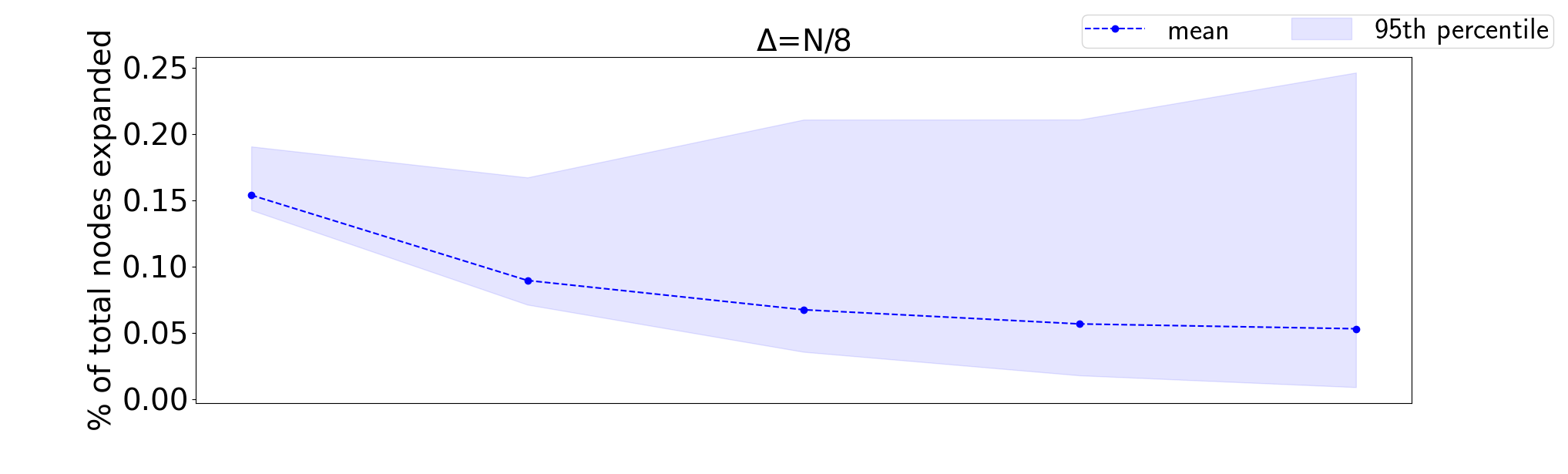}
\end{subfigure}
\hspace*{-1cm} \begin{subfigure}{0.55\textwidth}
\centering
\includegraphics[scale = 0.18]{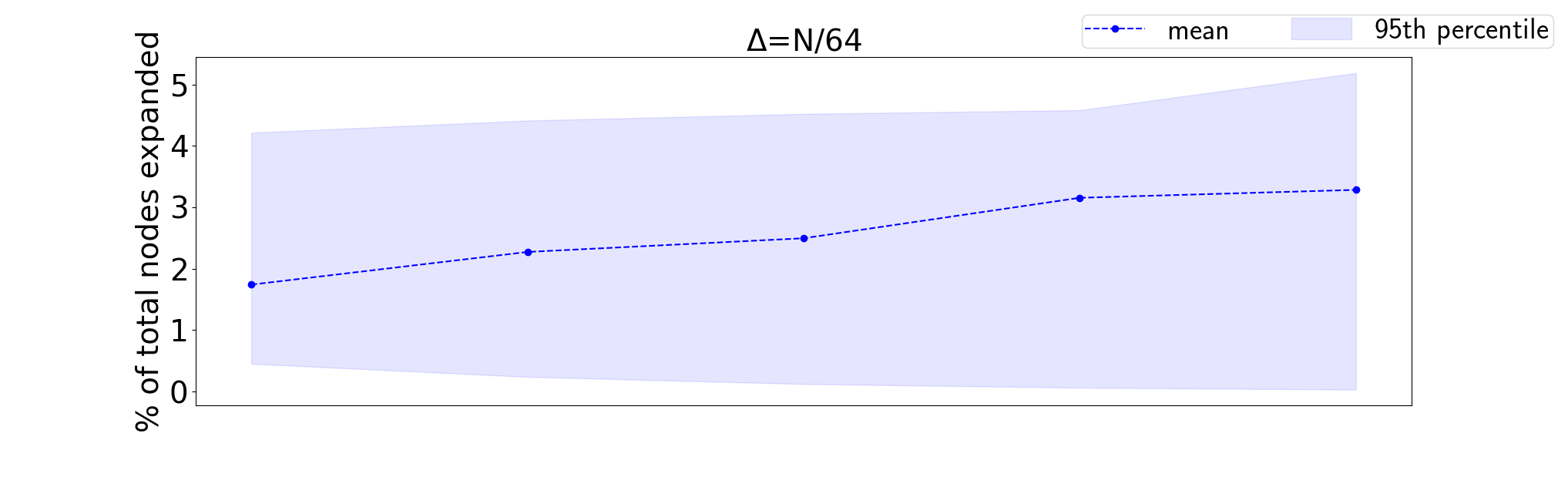}
\end{subfigure}
\begin{subfigure}{0.5\textwidth}
\centering
\includegraphics[scale = 0.18]{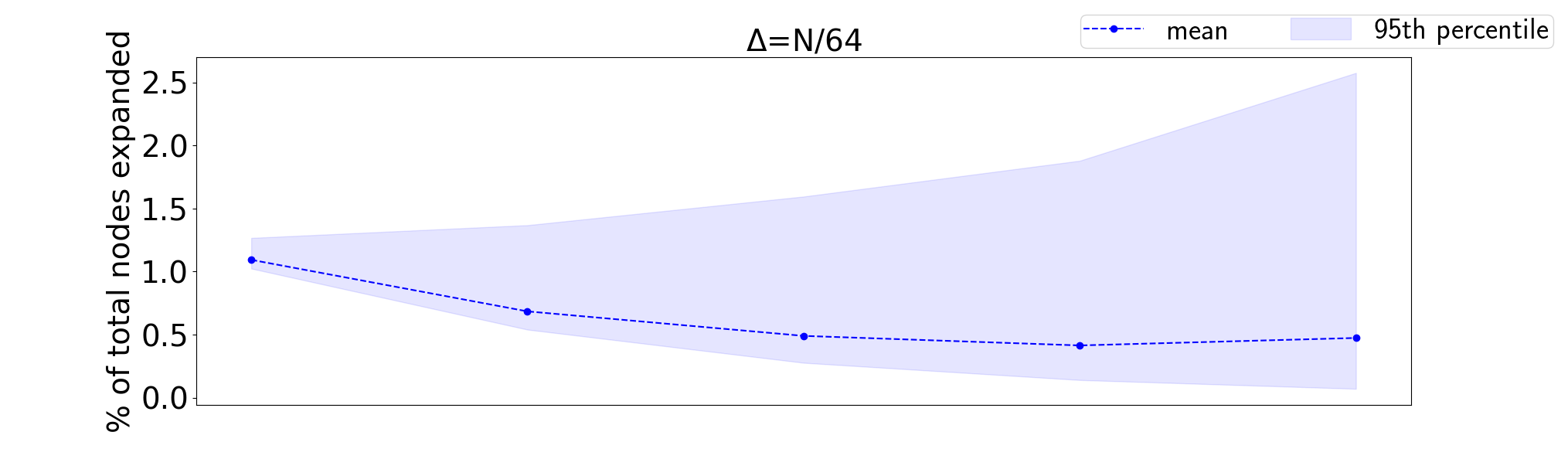}
\end{subfigure}
\hspace*{-1cm} \begin{subfigure}{0.55\textwidth}
\centering
\includegraphics[scale = 0.18]{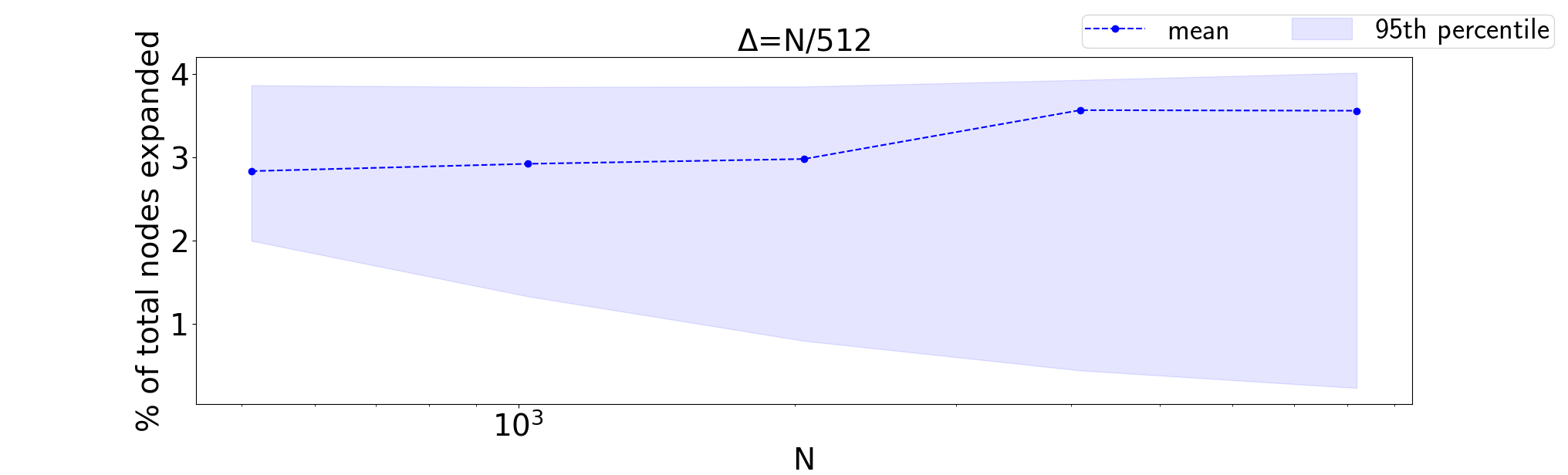}
\end{subfigure}
\begin{subfigure}{0.5\textwidth}
\centering
\includegraphics[scale = 0.18]{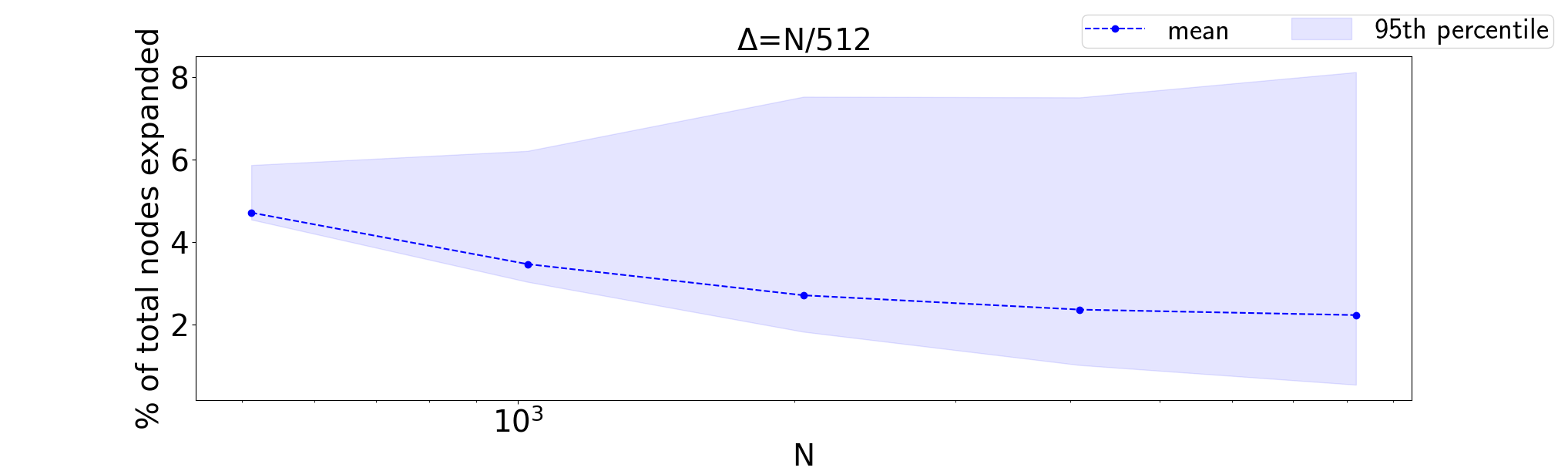}
\end{subfigure}
\caption{Mean percentage of nodes (blue, dashed) expanded by
\ref{itm:astar} for different values of $N$ for the graphs of
the \eqref{eq:ip} instances stemming from \eqref{eq:iocp_steady} (left)
and \eqref{eq:iocp_signal} (right). The lightly colored areas represent
the 95th percentile of expanded nodes, that is 5 percent of
the instances expand a larger fraction of the nodes in the graph.
The trust-region radii $\Delta$ are $\frac{1}{8} N$ (top),
$\frac{1}{64} N$ (center) and $\frac{1}{512} N$ (bottom).}
\label{fig:deviation}
\end{figure}

\begin{figure}[H]
\hspace*{-1cm}\begin{subfigure}{0.55\textwidth}
\centering
\includegraphics[scale = 0.18]{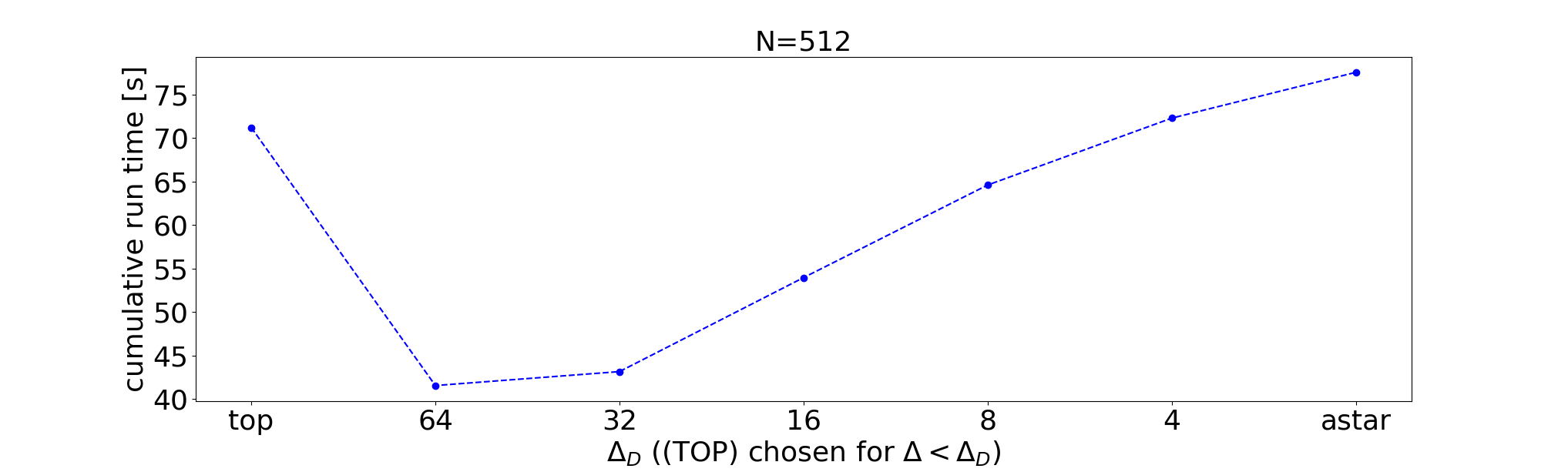}
\end{subfigure}
\begin{subfigure}{0.5\textwidth}
\centering
\includegraphics[scale = 0.18]{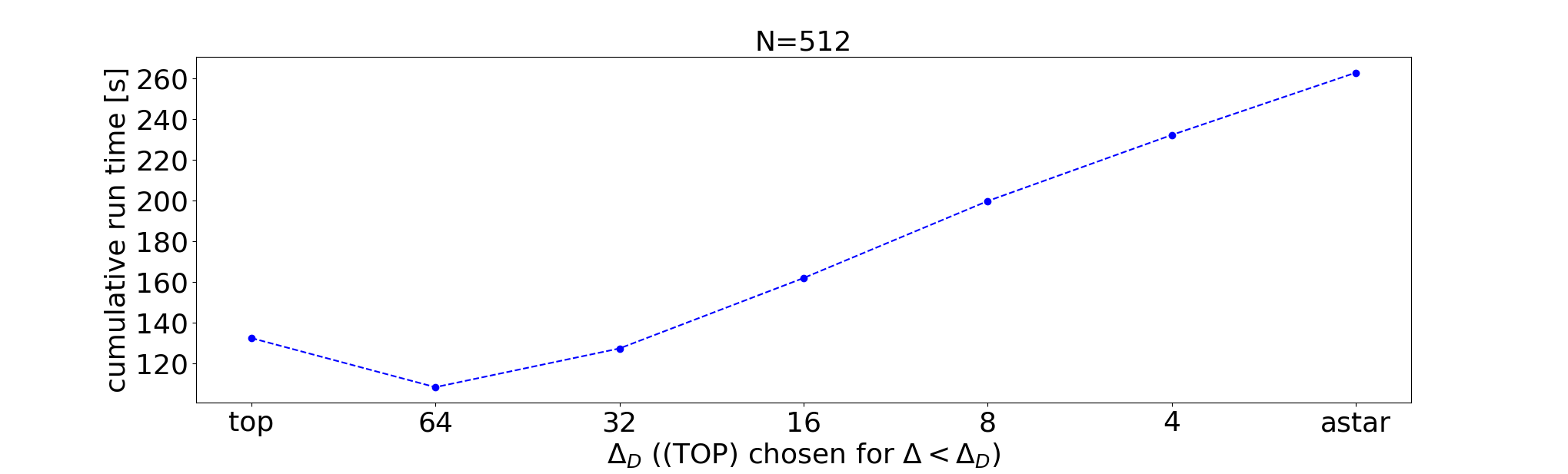}
\end{subfigure}
\hspace*{-1cm}\begin{subfigure}{0.55\textwidth}
\centering
\includegraphics[scale = 0.18]{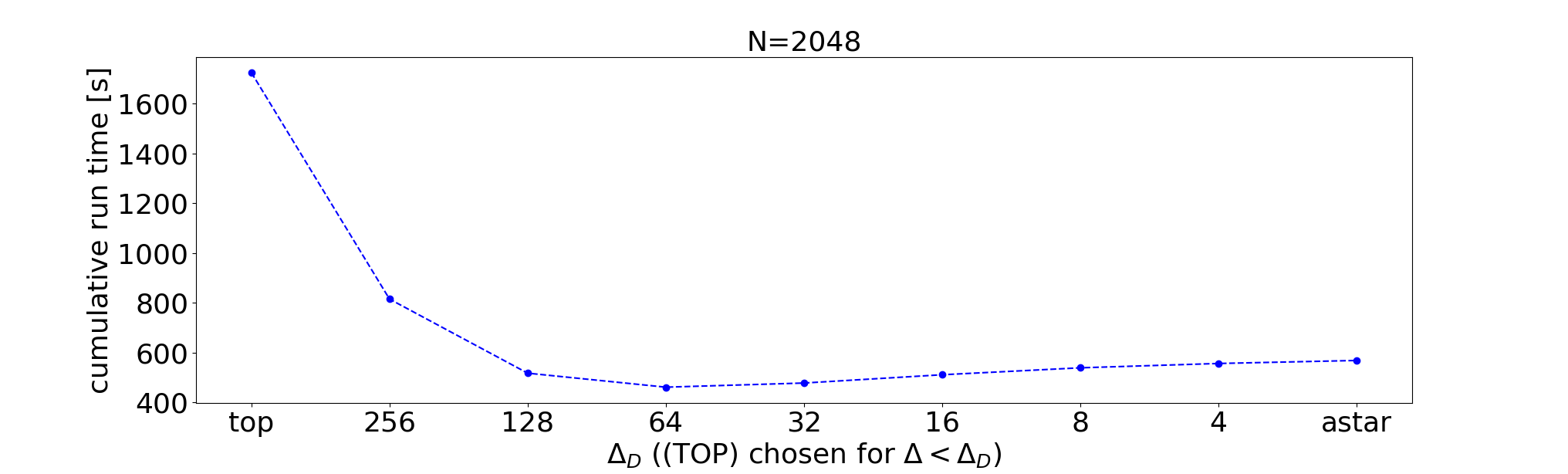}
\end{subfigure}
\begin{subfigure}{0.5\textwidth}
\centering
\includegraphics[scale = 0.18]{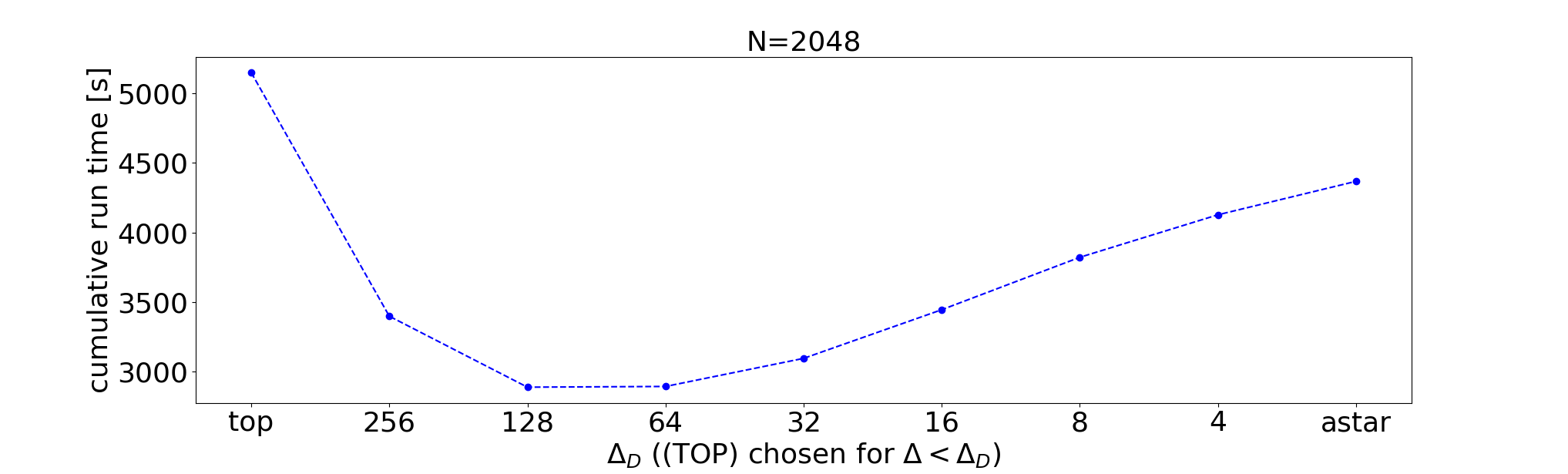}
\end{subfigure}
\hspace*{-1cm}\begin{subfigure}{0.55\textwidth}
\centering
\includegraphics[scale = 0.18]{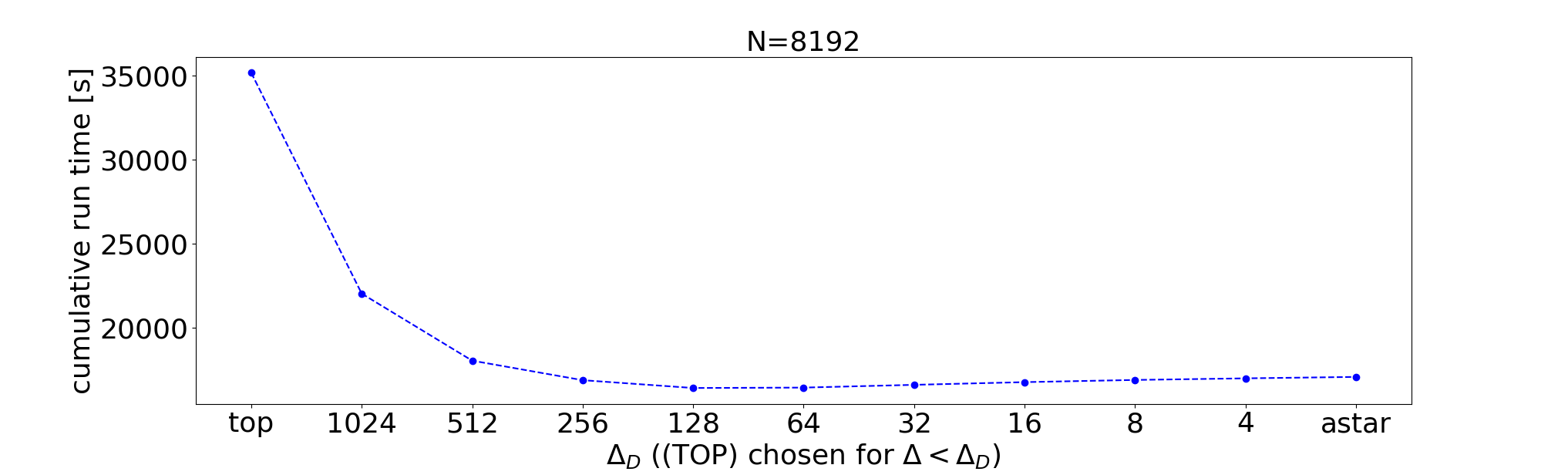}
\end{subfigure}
\begin{subfigure}{0.5\textwidth}
\centering
\includegraphics[scale = 0.18]{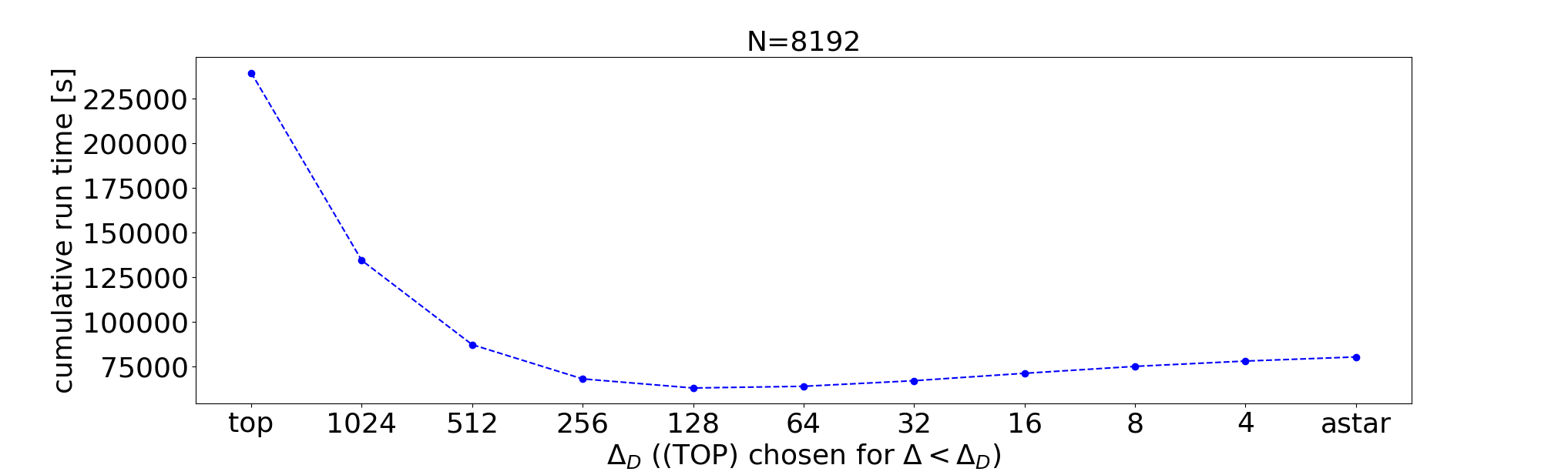}
\end{subfigure}
\caption{Cumulative run times for the subproblems of the form \eqref{eq:ip}
stemming from \eqref{eq:iocp_steady} (left) and \eqref{eq:iocp_signal}
(right) in the runs of \cref{alg:slip}
if one chooses \ref{itm:top} as subproblem solver for $\Delta < \Delta_D$
and else \ref{itm:astar} within \cref{alg:slip}.}
\label{fig:split}
\end{figure}

\bibliographystyle{apalike}
\bibliography{references}

\end{document}